\documentclass[hidelinks,onefignum,onetabnum]{siamart251216}

\usepackage{placeins}
\usepackage{microtype}
\usepackage{subfigure}
\usepackage{booktabs}
\usepackage{amsmath}
\usepackage{amssymb}
\usepackage{amsfonts}
\usepackage{mathtools}
\usepackage{algorithmic}
\usepackage{multirow}
\usepackage[table]{xcolor}
\usepackage{textcomp}
\usepackage[normalem]{ulem}
\usepackage[textsize=tiny]{todonotes}

\newsiamremark{remark}{Remark}
\newsiamremark{assumption}{Assumption}

\headers{PG-CFM with ERFM for Robust PDE Inverse Problems}{Y. Chen, S. Lu, W. Guo, and X. Zhong}

\title{Physics-Guided Conditional Flow Matching with Energy Regularization for Robust PDE Inverse Problems
 \thanks{Submitted to the editors May 2026. Code is available at \url{https://github.com/mosdf/PG-CFM-ERFM}.}
\funding{xxx.}}

\author{
Yongsheng Chen\thanks{School of Mathematical Sciences, Zhejiang University, Hangzhou 310027, China. (\email{22035024@zju.edu.cn}).\,\textsuperscript{*}Equal contribution.}
\and Shuo Lu\thanks{School of Mathematical Sciences, Zhejiang University, Hangzhou 310027, China. (\email{12435047@zju.edu.cn}).\,\textsuperscript{*}Equal contribution.}
\and Wei Guo\thanks{Department of Mathematics and Statistics, Texas Tech University, Lubbock, TX 70409, USA. (\email{weimath.guo@ttu.edu}).}
\and Xinghui Zhong\thanks{School of Mathematical Sciences, Zhejiang University, Hangzhou 310027, China. (\email{zhongxh@zju.edu.cn}).}
}

\ifpdf
\hypersetup{
  pdftitle={PG-CFM with ERFM for Robust PDE Inverse Problems},
  pdfauthor={Y. Chen, et al.}
}
\fi

\providecommand{\citep}[1]{\cite{#1}}
\providecommand{\citet}[1]{\cite{#1}}

\renewcommand{\mid}{\vert}

\newcommand{\best}[1]{\cellcolor{black!10}\textbf{#1}}

\newcommand{\figpanel}[4]{\begin{minipage}[t]{#1}
\centering
\includegraphics[width=\linewidth,height=#2,keepaspectratio]{#3}\\[-1mm]
{\footnotesize #4}
\end{minipage}
}

\renewcommand{\rm}[1]{\textcolor{green}{\sout{#1}}}

\begin{document}

\maketitle

% \begin{abstract}
% We study partial differential equation (PDE) inverse problems from sparse, noisy, and outlier-corrupted pointwise observations, where the goal is to recover unknown coefficient and solution fields. Existing physics-informed solvers and flow- or diffusion-based generative approaches are often ill-suited in this regime, as they typically treat observations indiscriminately, lacking a principled mechanism to reconcile physical laws with contaminated data. We introduce Physics-Guided Conditional Flow Matching (PG-CFM), a mesh-free conditional flow framework that enforces strong-form PDE consistency through residual regularization along flow trajectories, complemented by global collocation penalties. To robustly mitigate outliers, we propose Energy-Regularized Flow Matching (ERFM), a lightweight fine-tuning procedure that assigns physics–data energies to observations and reweights the flow-matching objective to downweight high-energy, PDE-inconsistent samples. We prove that ERFM is equivalent to conditional flow matching on a teacher-induced reweighted empirical distribution, providing a principled interpretation as distributional refocusing. Across multiple benchmarks with heterogeneous corruption and heavy-tailed noise, PG-CFM-ERFM consistently improves coefficient recovery over robust PINN variants and  generative baselines. Code is available at \url{}.
% \end{abstract}

%This regime
\begin{abstract}
We consider partial differential equation (PDE) inverse problems from sparse, noisy, and corrupted observations, with the aim of recovering unknown coefficient fields and associated state variables in a mesh-free setting. Under such sparse and corrupted observations, standard physics-informed and generative approaches typically treat all samples indiscriminately and therefore lack a principled mechanism for reconciling physical laws with contaminated data. We address this difficulty with a two-stage flow-matching framework. In the first stage, we develop physics-guided conditional flow matching (PG-CFM), which incorporates strong-form PDE information through residual regularization along the generative trajectories together with intermittent global collocation constraints. In the second stage, we introduce energy-regularized flow matching (ERFM), which fine-tunes the Stage-1 model by assigning each observation a physics--data energy score from a frozen teacher and reweighting the flow-matching objective to reduce the influence of high-energy, PDE-inconsistent samples. We show that the resulting Stage-2 objective is equivalent to flow matching under a teacher-induced reweighted data distribution, which gives a population-level interpretation of the robustness mechanism. Numerical experiments on several inverse benchmarks, including Poisson and Navier--Stokes problems, show that the proposed framework yields more accurate coefficient recovery than robust PINN variants and competing generative baselines.
\end{abstract}

\begin{keywords}
PDE inverse problems, conditional flow matching, physics-informed learning, robust learning
\end{keywords}

\begin{MSCcodes}
65M32, 65N21, 68T07
\end{MSCcodes}

\section{Introduction}
\label{sec:intro}

Partial differential equation (PDE) inverse problems seek to recover unknown coefficient fields, source terms, or latent state variables from observations  and arise in numerous applications throughout science and engineering, including subsurface flow, fluid reconstruction, among others. In practice, observations are inevitably sparse, irregularly distributed, and noisy, making the resulting inverse problem severely ill-posed. Classical variational and adjoint methods provide a solid foundation for such problems \citep{isakov2017inverse,arridge2019solving,benning2018modern}, but their performance can drop significantly under heteroscedastic noise.

A large body of recent work addresses PDE inverse problems through physics-informed optimization, most notably physics-informed neural networks (PINNs) and their variants \citep{raissi2019physics,karniadakis2021physics}. These methods combine data-misfit terms and PDE residual penalties into a single objective and yields a flexible mesh-free formulation. However, when observations are sparse or noisy, the data-fit and PDE-residual terms may induce imbalanced gradients, leading to inaccurate predictions \citep{wang2021understanding}. Robust PINN variants improve stability through adaptive weighting \citep{xiang2022self}, data-guided pretraining \citep{zhou2024data}, or Bayesian formulations \citep{yang2021b}. CoPINN dynamically evaluates the difficulty of each sample and optimizes the sampling regions from easy to hard \citep{duan2025copinn}. These approaches improve optimization robustness in different ways, but still act primarily through local loss shaping, training schedules, or posterior inference rather than through a mechanism that identifies observations that are globally inconsistent with the governing PDE.

Generative formulations provide an alternative by learning distributions over latent fields through diffusion, score-based, normalizing-flow, or flow-matching dynamics \citep{ho2020denoising,dhariwal2021diffusion,song2021scorebased,chen2018neural,lipman2023flow}. In the PDE setting, many recent approaches operate in data-rich offline regimes, where large collections of clean simulations are used to pretrain surrogates or priors \citep{li2025videopde,zhou2025text2pde,ye2025pdeformer}. Physical constraints may then be added at inference time, for example through diffusion posterior sampling \citep{huang2024diffusionpde,chung2023diffusion,yao2025guided}, or incorporated more directly into training through physics-aware generative objectives \citep{bastek2025physics,tauberschmidt2025physics,yuan2025pirf,baldan2025flow}. Functional flow-matching formulations further expand the admissible constraint mechanisms \citep{cheng2025gradientfree,utkarsh2025physics,ben2024d}. These developments are promising, but they do not naturally fit the single-instance inverse problem we study here: many assume relatively mild noise, rely on a fixed grid, or aim to balance data fit and physics, rather than down-weighting corrupted measurements when the two goals are in clear conflict.

A related issue appears more broadly in learning with noisy data. Robustness is often improved through robust losses \citep{ghosh2017robust,zhang2018generalized}, sample selection or co-teaching \citep{xia2022sample,song2022learning}, meta-reweighting \citep{ren2018learning,zhang2020distilling}, influence-based criteria \citep{basu2021influence,liu2025rethinking}, or machine-unlearning mechanisms that reduce the footprint of harmful samples \citep{pruthi2020estimating,liu2025threats,simone2025continual}. Such strategies are effective in many statistical settings, but they are largely physics-agnostic: a sample that appears statistically reasonable may still violate the governing PDE, while a large residual may indicate true model mismatch rather than corrupted data. This motivates a robustness mechanism guided by data fit and physical consistency.

We consider this problem in the online inverse setting, where one solves a single PDE inverse problem directly from irregular observations without relying on large offline pretraining datasets. Our formulation is built on conditional flow matching and consists of two stages. The first stage, \emph{physics-guided conditional flow matching} (PG-CFM), augments conditional flow matching with strong-form PDE regularization along the generative trajectories together with periodic global collocation penalties. This encourages the transport aligned with the PDE solution manifold throughout the flow, not just at the final reconstruction. The second stage, \emph{energy-regularized flow matching} (ERFM), uses the Stage-1 solution as a frozen teacher, assigns each observation a physics--data energy score, and reweights the flow-matching objective to reduce the influence of high-energy, PDE-inconsistent samples.

The main theoretical result shows that the ERFM objective is equivalent to conditional flow matching under a teacher-induced reweighted data distribution, which gives a population-level interpretation of the Stage-2 robustness mechanism. In the numerical experiments, we evaluate the method on several inverse benchmarks, including Poisson and Navier--Stokes problems, and obtain improved coefficient recovery and field reconstruction under heterogeneous corruption and heavy-tailed noise.

The remainder of the paper is organized as follows. 
We formulate the PDE inverse problem and review the background of conditional flow matching in Section~\ref{sec:problem-background}.
Section~\ref{sec:pgcfm-framework} introduces the proposed physics-guided conditional flow-matching framework.
Section~\ref{sec:erfm} introduces the energy-regularized refinement stage.
We report numerical experiments in Section~\ref{sec:experiments}, including benchmark comparisons, ablation studies, robustness tests, sensitivity analyses, and computational diagnostics. Finally, conclusion and future work are discussed in Section~\ref{sec:conclusion}.

\section{Problem Formulation and Flow-Matching Background}
\label{sec:problem-background}

\subsection{PDE inverse problems}
\label{sec:background-pde}

We consider PDE inverse problems on a physical domain
$\Omega \subset \mathbb{R}^{d_\xi}$, where a coordinate
$\boldsymbol{\xi}\in\Omega$ collects all relevant independent variables. For time-dependent problems, one may take
$\Omega=\Omega_{\mathrm{space}}\times[0,T]$, although we do not distinguish these coordinates explicitly in the notation.

The goal is to recover unknown state and coefficient fields
$u^\star:\Omega \to \mathbb{R}^{d_u}$ and
$a^\star:\Omega \to \mathbb{R}^{d_a}$ that satisfy the governing PDE together with the associated boundary or initial conditions,
\begin{equation}
\mathcal{F}\bigl(u^\star,a^\star;\boldsymbol{\xi}\bigr)=\mathbf{0},
\qquad \boldsymbol{\xi}\in\Omega,
\label{eq:general-pde}
\end{equation}
and
\begin{equation}
\mathcal{B}\bigl(u^\star,a^\star;\boldsymbol{\xi}\bigr)=\mathbf{0},
\qquad \boldsymbol{\xi}\in\partial\Omega.
\label{eq:general-bc}
\end{equation}
Here $\mathcal{F}$ is a differential operator involving $u^\star$, $a^\star$, and their derivatives, while $\mathcal{B}$ represents the prescribed boundary and initial constraints.

In the inverse setting, the field $u^\star$ is only partially observed. Instead, we are given $N$ noisy, potentially corrupted measurements
\[
y_i = u^\star(\boldsymbol{\xi}_i) + \varepsilon_i,
\quad \boldsymbol{\xi}_i\in\Omega,\quad i=1,\dots,N,
\]
collected at scattered locations $\boldsymbol{\xi}_i$, where $\varepsilon_i$ may be heteroscedastic, heavy-tailed, or include outliers. This setting calls for a mesh-free reconstruction approach that can work directly with irregular observations. For notational convenience, we write
\[
(\boldsymbol{\xi},y)\sim p_{\mathrm{data}}
\]
for the data-generating distribution and approximate expectations under $p_{\mathrm{data}}$ by empirical averages over $\{(\boldsymbol{\xi}_i,y_i)\}_{i=1}^N$.

Our objective is to reconstruct both $u^\star$ and $a^\star$ from these noisy and corrupted observations while enforcing consistency with \eqref{eq:general-pde}--\eqref{eq:general-bc}. Throughout the paper, we focus on the single-instance setting in which one solves a given inverse problem directly from its irregular measurements, rather than learning from a large offline dataset of fully observed PDE solutions.

\subsection{Conditional flow-matching}
\label{sec:problem-cfm}
We briefly review the conditional flow-matching (CFM) formulation used below; a general treatment can be found in \citet{lipman2023flow}. Let $\{p_t\}_{t\in[0,1]}$ be a probability path on $\mathbb{R}^{d_u}$ from a simple base distribution $p_0$ to a target distribution $p_1$. The path is governed with a time-dependent velocity field
\[
w_t:\mathbb{R}^{d_u}\to\mathbb{R}^{d_u}
\]
through the continuity equation
\begin{equation}
\partial_t p_t(u) + \nabla_u\!\cdot\!\bigl(p_t(u)\,w_t(u)\bigr)=0,
\qquad u\in\mathbb{R}^{d_u}.
\label{eq:cfm-continuity}
\end{equation}

A neural approximation $v_t(u;\theta)$ to the marginal velocity $w_t(u)$ can be learned by minimizing
\begin{equation}
\mathcal{L}_{\mathrm{FM}}(\theta)
:=
\mathbb{E}_{t\sim\mathrm{Unif}[0,1],\,u\sim p_t}
\bigl[\|v_t(u;\theta)-w_t(u)\|_2^2\bigr].
\label{eq:fm-loss-marginal}
\end{equation}
However, the marginal velocity $w_t$ is generally intractable. CFM addresses this difficulty by introducing a conditioning variable $z\sim q(z)$ and a family of conditional paths $\{p_t(\cdot\mid z)\}_{t\in[0,1]}$ such that
\begin{equation}
p_t(u)=\int p_t(u\mid z)\,q(z)\,dz,
\label{eq:cfm-marginalization}
\end{equation}
and each conditional path satisfies
\begin{equation}
\partial_t p_t(u\mid z)
+
\nabla_u\!\cdot\!\bigl(p_t(u\mid z)\,w_t(u\mid z)\bigr)=0.
\label{eq:cfm-conditional-path}
\end{equation}
When the conditional velocity $w_t(u\mid z)$ is tractable, $v_t$ can instead be trained by
\begin{equation}
\mathcal{L}_{\mathrm{CFM}}(\theta)
:=
\mathbb{E}_{\substack{t\sim\mathrm{Unif}[0,1]\\ z\sim q(z),\,u\sim p_t(\cdot\mid z)}}
\Bigl[
\|v_t(u;\theta)-w_t(u\mid z)\|_2^2
\Bigr].
\label{eq:cfm-loss}
\end{equation}
In practice, we sample $t$ from $\mathrm{Unif}[\varepsilon_{\mathrm{fm}},1-\varepsilon_{\mathrm{fm}}]$ with a small $\varepsilon_{\mathrm{fm}}>0$ to avoid numerical issues near the endpoints. Under mild regularity assumptions, the conditional objective \eqref{eq:cfm-loss} shares the same minimizers as marginal flow matching and its gradients differ only by a constant factor \citep{lipman2023flow}. CFM therefore allows one to train continuous transport models without explicitly simulating the marginal velocity.

\subsection{Observation-conditioned path}
\label{sec:cfm-obs}

To apply CFM to the inverse problems considered here, we associate each observation pair $(\boldsymbol{\xi}_i,y_i)$ with a terminal target in state space. We model the flow with a neural vector field
\[
v_\theta:[0,1]\times\mathbb{R}^{d_u}\times\Omega\times\mathbb{R}^{d_a}\to\mathbb{R}^{d_u},
\]
which takes as inputs the flow time $t$, the current state $u$, the coordinate $\boldsymbol{\xi}$, and the coefficient value $a_\Phi(\boldsymbol{\xi})$. When problem-dependent inputs such as forcing terms are used, they are included as additional conditioning variables and are omitted from the general notation for readability. The unknown coefficient field is approximated by a neural representation
\[
a_\Phi:\Omega\to\mathbb{R}^{d_a},
\]
with trainable parameters $\Phi$.

For a fixed coordinate $\boldsymbol{\xi}\in\Omega$, the induced transport satisfies
\begin{equation}
\frac{d s(t,\boldsymbol{\xi})}{dt}
=
v_\theta\bigl(t,s(t,\boldsymbol{\xi}),\boldsymbol{\xi},a_\Phi(\boldsymbol{\xi})\bigr),
\qquad
s(0,\boldsymbol{\xi})\sim p_0,
\label{eq:conditional-flow-pde}
\end{equation}
where $p_0$ is a standard Gaussian on $\mathbb{R}^{d_u}$. For each observed pair $(\boldsymbol{\xi}_i,y_i)$, we sample $\epsilon\sim p_0$ and $t\sim \mathrm{Unif}[0,1]$, and define the linear conditional path
\begin{equation}
\bar{s}(t,\boldsymbol{\xi}_i)
=
(1-t)\epsilon + t y_i,
\qquad t\in[0,1].
\label{eq:linear-bridge}
\end{equation}
For this standard CFM linear bridge, the conditional velocity is
\begin{equation}
w_t\bigl(\bar{s}(t,\boldsymbol{\xi}_i)\mid \boldsymbol{\xi}_i,y_i,\epsilon\bigr)=y_i-\epsilon.
\label{eq:conditional-velocity}
\end{equation}
The resulting observation-conditioned CFM objective is
\begin{equation}
\mathcal{L}_{\mathrm{CFM}}(\theta,\Phi)
:=
\mathbb{E}_{\substack{t\sim\mathrm{Unif}[0,1]\\ (\boldsymbol{\xi}_i,y_i)\sim p_{\mathrm{data}}\\ \epsilon\sim p_0}}
\Bigl[
\bigl\|
v_\theta\bigl(t,\bar{s}(t,\boldsymbol{\xi}_i),\boldsymbol{\xi}_i,a_\Phi(\boldsymbol{\xi}_i)\bigr)
-
\bigl(y_i-\epsilon\bigr)
\bigr\|_2^2
\Bigr].
\label{eq:fm-pde-inverse}
\end{equation}

By itself, \eqref{eq:fm-pde-inverse} learns a conditional transport model from noisy observations, but it  does not use the governing PDE or distinguish physically consistent measurements from corrupted ones. The next section introduces a physics-guided extension of \eqref{eq:fm-pde-inverse} that regularizes the transport dynamics with strong-form PDE information.
\section{Physics-Guided Conditional Flow Matching}
\label{sec:pgcfm-framework}

The CFM objective \eqref{eq:fm-pde-inverse} learns a transport model from noisy observations, but does not enforce the governing PDE. Consequently, fitting the bridge velocities alone does not ensure that the learned trajectories remain physically consistent, especially when the observations are sparse or corrupted. We therefore augment CFM with PDE-based regularization and obtain the first-stage model, which we call \emph{physics-guided conditional flow matching} (PG-CFM).

Figure~\ref{fig:framework} summarizes the complete two-stage procedure. This section develops the Stage-1 PG-CFM model, where observation-conditioned flow matching is combined with bridge-local and global PDE residual penalties. The Stage-2 energy-based refinement, which uses the frozen Stage-1 solution to reweight the flow-matching objective, is introduced in Section~\ref{sec:erfm}.

\begin{figure}[t]
    \centering
    \includegraphics[width=\linewidth]{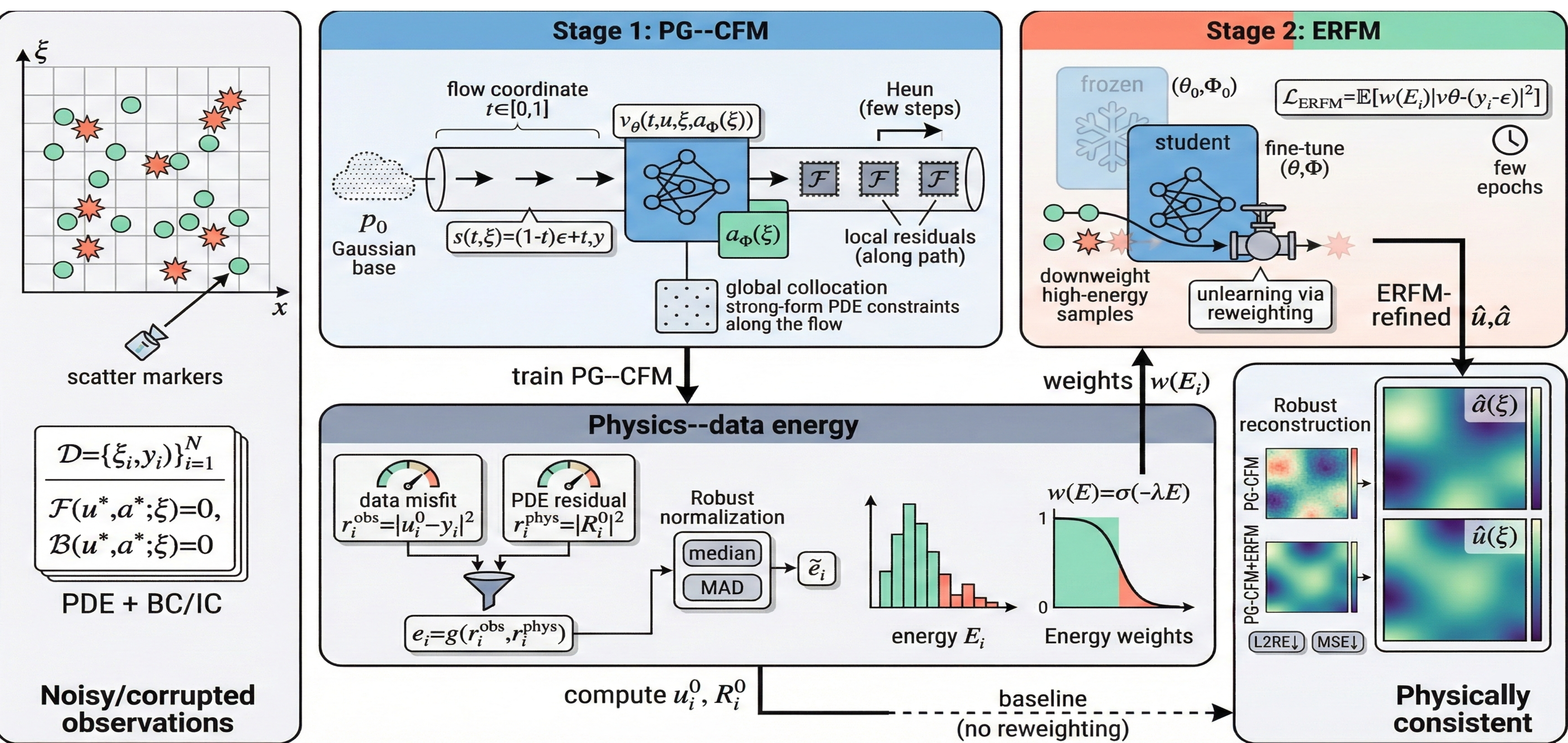}
    \caption{Overview of the proposed two-stage framework. In Stage~1, PG-CFM learns a physics-guided conditional transport model by combining observation-conditioned flow matching with local and global PDE regularization. In Stage~2, the frozen Stage-1 solution induces physics--data energy scores, which reweight the observation-conditioned flow-matching loss for ERFM fine-tuning.}
    \label{fig:framework}
\end{figure}

\subsection{Physics-guided regularization}
\label{sec:pg-cfm}

To incorporate physical structure, we add residual penalties built from the operator in \eqref{eq:general-pde}. For a candidate state field $\hat u$ and coefficient field $a$, define
\begin{equation}
\mathcal{R}[\hat u,a](\boldsymbol{\xi})
:=
\mathcal{F}(\hat u,a;\boldsymbol{\xi}),
\label{eq:pde-residual-operator}
\end{equation}
where the derivatives in $\mathcal{F}$ are evaluated by automatic differentiation \citep{baydin2018automatic}.

We use two residual penalties. The first is a local bridge-conditioned loss evaluated along trajectories attached to the noisy observations. The second is a global collocation loss imposed across the domain. The local term constrains the transport near the observed data while the global term promotes consistency in regions.

\subsubsection{Local bridge-conditioned physics loss}
\label{sec:local-physics-loss}
The local bridge-conditioned physics loss is designed to enforce PDE consistency on terminal predictions generated from bridge states. Let $\Gamma_{t\to 1}^{(K_1)}$ denote the numerical flow map obtained by applying $K_1$ fixed-step Heun updates to \eqref{eq:conditional-flow-pde} from flow time $t$ to $1$. Starting from the bridge state $\bar{s}(t,\boldsymbol{\xi}_i)$, we define
\begin{equation}
\tilde{u}_{\theta,\Phi}(\boldsymbol{\xi}_i)
=
\Gamma_{t\to 1}^{(K_1)}
\bigl(
\bar{s}(t,\boldsymbol{\xi}_i),
\boldsymbol{\xi}_i,
a_\Phi(\boldsymbol{\xi}_i);\theta
\bigr).
\label{eq:local-flow-to-field}
\end{equation}
The corresponding local residual is
\begin{equation}
\mathcal{R}_{\mathrm{loc}}(\boldsymbol{\xi}_i,t;\theta,\Phi)
=
\mathcal{R}
[
\tilde{u}_{\theta,\Phi},
a_\Phi
]
(\boldsymbol{\xi}_i),
\label{eq:local-residual}
\end{equation}
and the local physics loss is
\begin{equation}
\mathcal{L}_{\mathrm{PDE,loc}}(\theta,\Phi)
=
\mathbb{E}_{
\substack{
t\sim\mathrm{Unif}[0,1],\,
(\boldsymbol{\xi}_i,y_i)\sim p_{\mathrm{data}},\\
\epsilon\sim p_0
}}
\left[
\left\|
\mathcal{R}_{\mathrm{loc}}(\boldsymbol{\xi}_i,t;\theta,\Phi)
\right\|_2^2
\right].
\label{eq:local-pde-loss}
\end{equation}

Here, ``local'' means that the PDE residual is evaluated on a terminal prediction obtained from an observation-conditioned intermediate bridge state. Specifically, for each sampled observation $(\boldsymbol{\xi}_i,y_i)$ and flow time $t$, the bridge state $\bar{s}(t,\boldsymbol{\xi}_i)$ is used as the initial condition for the learned vector field, which is then integrated from $t$ to $1$. This differs from generating a prediction by integrating a full trajectory from a base state $s_0$ over $[0,1]$ and reduces the mismatch between the bridge interpolation used for flow matching and the physical field on which the residual is evaluated.

\subsubsection{Global collocation physics loss}
\label{sec:global-physics-loss}

The local loss alone does not sufficiently constrain the recovered solution away from the observation set. We therefore sample collocation points $\boldsymbol{\xi}\sim p_{\mathrm{res}}$ over $\Omega$, where $p_{\mathrm{res}}$ denotes the collocation-point sampling distribution, and impose the PDE residual independently of the measurements. Let $\Gamma_{0\to 1}^{(K_2)}$ denote the numerical flow map obtained by applying $K_2$ fixed-step Heun updates over $[0,1]$. We define the deterministic global prediction
\begin{equation}
u^\dagger_{\theta,\Phi}(\boldsymbol{\xi})
=
\Gamma_{0\to 1}^{(K_2)}
\bigl(
u_{\mathrm{base}},
\boldsymbol{\xi},
a_\Phi(\boldsymbol{\xi});\theta
\bigr),
\qquad
u_{\mathrm{base}}=\mathbf{0}.
\label{eq:global-flow-to-field}
\end{equation}
The deterministic initialization in \eqref{eq:global-flow-to-field} gives a single differentiable field value at each collocation point, which is convenient for evaluating the global residual and reporting point predictions.

The corresponding global residual is
\begin{equation}
\mathcal{R}_{\mathrm{glob}}(\boldsymbol{\xi};\theta,\Phi)
=
\mathcal{R}
[
u^\dagger_{\theta,\Phi},
a_\Phi
]
(\boldsymbol{\xi}),
\label{eq:global-residual}
\end{equation}
and the global physics loss is
\begin{equation}
\mathcal{L}_{\mathrm{PDE}}(\theta,\Phi)
=
\mathbb{E}_{\boldsymbol{\xi}\sim p_{\mathrm{res}}}
\left[
\left\|
\mathcal{R}_{\mathrm{glob}}(\boldsymbol{\xi};\theta,\Phi)
\right\|_2^2
\right].
\label{eq:pde-loss-pgcfm}
\end{equation}
This term promotes domain-wide consistency and encourages the inferred state and coefficient fields to satisfy the governing PDE away from the measurement locations. In practice, it may be evaluated intermittently to reduce computational cost.

\subsection{Boundary and initial conditions}
\label{sec:pgcfm-bc}

When boundary or initial conditions are available, we impose them through additional squared penalties. Let $\partial\Omega_u\subseteq\Omega$ denote the portion of the spatiotemporal boundary on which the state is prescribed, and let $u_{\mathrm{bc}}(\boldsymbol{\xi})$ denote the corresponding boundary or initial data. We define
\begin{equation}
\mathcal{L}_{\mathrm{BC}}^u(\theta,\Phi)
=
\mathbb{E}_{\boldsymbol{\xi}\sim p_{\partial\Omega_u}}
\left[
\left\|
\mathcal{B}_u
[
u^\dagger_{\theta,\Phi}
]
(\boldsymbol{\xi})
-
u_{\mathrm{bc}}(\boldsymbol{\xi})
\right\|_2^2
\right],
\label{eq:bc-loss-u}
\end{equation}
where $\mathcal{B}_u$ is the boundary or initial-condition operator for the state. For time-dependent problems, the initial condition is imposed by including the temporal boundary, such as $\{t_{\mathrm{phys}}=0\}$, in $\partial\Omega_u$.

If the coefficient field is known on a subset $\partial\Omega_a\subseteq\Omega$, we also use
\begin{equation}
\mathcal{L}_{\mathrm{BC}}^a(\Phi)
=
\mathbb{E}_{\boldsymbol{\xi}\sim p_{\partial\Omega_a}}
\left[
\left\|
a_\Phi(\boldsymbol{\xi})
-
a_{\mathrm{bc}}(\boldsymbol{\xi})
\right\|_2^2
\right].
\label{eq:bc-loss-a}
\end{equation}
When such information is unavailable, the corresponding penalty weight is set to zero.

\subsection{PG-CFM objective and stage-1 reconstruction}
\label{sec:pgcfm-objective}
Combining the CFM loss with the local and global physics regularizers gives the PG-CFM objective
\begin{equation}
\begin{split}
\mathcal{L}_{\mathrm{PG\mbox{-}CFM}}(\theta,\Phi)
= {}& \mathcal{L}_{\mathrm{CFM}}(\theta,\Phi)
      + \lambda_{\mathrm{loc}} \mathcal{L}_{\mathrm{PDE,loc}}(\theta,\Phi) \\
    & + \lambda_{\mathrm{PDE}} \mathcal{L}_{\mathrm{PDE}}(\theta,\Phi)
      + \lambda_{\mathrm{BC}}^u \mathcal{L}_{\mathrm{BC}}^u(\theta,\Phi)
      + \lambda_{\mathrm{BC}}^a \mathcal{L}_{\mathrm{BC}}^a(\Phi).
\end{split}
\label{eq:pg-cfm-loss}
\end{equation}
Here $\lambda_{\mathrm{loc}}$ and $\lambda_{\mathrm{PDE}}$ weight the local and global physics penalties, while $\lambda_{\mathrm{BC}}^u$ and $\lambda_{\mathrm{BC}}^a$ weight the penalties for prescribed state and coefficient boundary or initial data.

In implementation, $\mathcal{L}_{\mathrm{CFM}}$ and $\mathcal{L}_{\mathrm{PDE,loc}}$ are evaluated on the same data minibatches, while $\mathcal{L}_{\mathrm{PDE}}$ is computed on collocation points sampled from $p_{\mathrm{res}}$. The coefficient network $a_\Phi$ and the conditional vector field $v_\theta$ are trained jointly by minimizing \eqref{eq:pg-cfm-loss}. Thus, $a_\Phi$ is updated both through its role in the flow model and its appearance in the PDE residual.

Stage~1 minimizes \eqref{eq:pg-cfm-loss} to obtain a physics-guided baseline model. The two residual terms play different roles: $\mathcal{L}_{\mathrm{PDE,loc}}$ constrains the transport near the observation-conditioned bridges, while $\mathcal{L}_{\mathrm{PDE}}$ promotes broader consistency across the domain. Using both is helpful in sparse inverse problems, where neither local observation information nor global collocation information is sufficient on its own.

After training, the model yields an inferred coefficient field $a_\Phi$ together with a transport field $v_\theta$ that defines the recovered state. At inference time, point predictions are obtained from the deterministic map $u^\dagger_{\theta,\Phi}$ in \eqref{eq:global-flow-to-field}. The pair $(u^\dagger_{\theta,\Phi},a_\Phi)$ is the Stage-1 inverse solution and also serves as the frozen teacher for the Stage-2 refinement in Section~\ref{sec:erfm}.

\begin{remark}[Numerical integration]
\label{rem:pgcfm-integration}
The flow maps $\Gamma_{t\to 1}^{(K_1)}$ and $\Gamma_{0\to 1}^{(K_2)}$ are computed using fixed-step Heun integration in the default experiments. The same integration scheme is used during training and evaluation. For the local physics term, the bridge-conditioned residual requires integration only from an intermediate flow time $t$ to the terminal time $1$, which is typically cheaper than repeatedly evolving a full trajectory over $[0,1]$.
\end{remark}

Algorithm~\ref{alg:pg-cfm} summarizes one PG-CFM training step, including the observation-conditioned flow-matching term, the local and global physics penalties, and the optional boundary or initial-condition terms.

\begin{algorithm}[t]
\caption{Stage~1: PG-CFM training step (one minibatch)}
\label{alg:pg-cfm}
\begin{algorithmic}[1]
  \STATE \textbf{Input:} minibatch $(\boldsymbol{\xi},y)$
  \STATE sample $t \sim \mathrm{Unif}[\varepsilon_{\mathrm{fm}},1-\varepsilon_{\mathrm{fm}}]$, $\epsilon \sim p_0$
  \STATE $s_t \leftarrow (1-t)\epsilon + ty$ \hfill \# linear bridge
  \STATE $u_{\mathrm{loc}} \leftarrow \Gamma_{t\to1}^{(K_1)}(s_t,\boldsymbol{\xi},a_\Phi(\boldsymbol{\xi});\theta)$
  \STATE $\mathcal{L}_{\mathrm{FM}} \leftarrow \mathrm{FM\_loss}(v_\theta,s_t,y,\epsilon)$ \hfill \# \eqref{eq:fm-pde-inverse}
  \STATE $\mathcal{L}_{\mathrm{PDE,loc}} \leftarrow \mathrm{LocalPDE\_loss}(u_{\mathrm{loc}},a_\Phi,\boldsymbol{\xi})$ \hfill \# \eqref{eq:local-pde-loss}
  \STATE $\mathcal{L}_{\mathrm{PDE}} \leftarrow \mathrm{GlobalPDE\_loss}(u^\dagger_{\theta,\Phi},a_\Phi)$ \hfill \# \eqref{eq:pde-loss-pgcfm}, evaluated occasionally
  \STATE $\mathcal{L}_{\mathrm{PG\mbox{-}CFM}} \leftarrow \mathcal{L}_{\mathrm{CFM}} + \lambda_{\mathrm{loc}}\mathcal{L}_{\mathrm{PDE,loc}} + \lambda_{\mathrm{PDE}}\mathcal{L}_{\mathrm{PDE}}$
  \STATE \textbf{if} boundary or initial penalties are used \textbf{then}
  \STATE \hspace{1.5em} $\mathcal{L}_{\mathrm{PG\mbox{-}CFM}} \leftarrow \mathcal{L}_{\mathrm{PG\mbox{-}CFM}} + \lambda_{\mathrm{BC}}^u\mathcal{L}_{\mathrm{BC}}^u + \lambda_{\mathrm{BC}}^a\mathcal{L}_{\mathrm{BC}}^a$
  \STATE update $(\theta,\Phi)$ by one optimizer step on $\mathcal{L}_{\mathrm{PG\mbox{-}CFM}}$
\end{algorithmic}
\end{algorithm}
\section{Energy-Regularized Flow Matching}
\label{sec:erfm}

The Stage-1 PG-CFM objective in \eqref{eq:pg-cfm-loss} weights all observations uniformly. In practice, however, some measurements may be corrupted and inconsistent with the governing PDE. To address this, we introduce \emph{energy-regularized flow matching} (ERFM), a second-stage refinement that assigns each observation a physics-data energy score from the frozen Stage-1 model and reweights the observation-conditioned flow-matching loss accordingly.

\subsection{Physics-data energy score}
\label{sec:erfm-energy}

Let $(u_{\theta^{0},\Phi^{0}}^\dagger,a_{\Phi^{0}})$ denote the Stage-1 solution obtained from PG-CFM. In Stage~2, this solution is kept fixed. For each observation $(\boldsymbol{\xi}_i,y_i)$, we evaluate the deterministic PG-CFM prediction and the corresponding PDE residual at the same location:
\begin{equation}
u_i^0
:=
u^\dagger_{\theta^0,\Phi^0}(\boldsymbol{\xi}_i),
\qquad
\mathcal{R}_i^0
:=
\mathcal{F}
\bigl(
u_i^0,
a_{\Phi^0}(\boldsymbol{\xi}_i);
\boldsymbol{\xi}_i
\bigr).
\end{equation}
We then define the observation misfit and physics residual magnitudes by
\begin{equation}
r_i^{\mathrm{obs}}
:=
\|u_i^0-y_i\|_2,
\qquad
r_i^{\mathrm{phys}}
:=
\|\mathcal{R}_i^0\|_2 .
\end{equation}
These two quantities are combined into a raw error score
\begin{equation}
e_i
:=
\omega_{\mathrm{obs}} r_i^{\mathrm{obs}}
+
\omega_{\mathrm{phys}} r_i^{\mathrm{phys}},
\label{eq:raw-error}
\end{equation}
where $\omega_{\mathrm{obs}},\omega_{\mathrm{phys}}\ge 0$ control the relative importance of data misfit and PDE violation.

We then normalize $\{e_i\}_{i=1}^N$ using the median and median absolute deviation to obtain a dimensionless robust score. Let
\[
m
:=
\operatorname{median}_{1\le i\le N} e_i,
\qquad
\operatorname{MAD}
:=
\operatorname{median}_{1\le i\le N} |e_i-m|.
\]
We set
\[
\mu := m+\kappa \operatorname{MAD},
\qquad
\tilde e_i
:=
\frac{e_i-\mu}{\operatorname{MAD}  +\delta},
\]
where $\kappa\ge 0$ controls the threshold at which samples are treated as atypical, and $\delta>0$ is a small stability constant. The final energy score is
\begin{equation}
E_i
:=
\rho(\tilde e_i),
\qquad
\rho(x)=\frac{x}{1+|x|}.
\label{eq:energy-score}
\end{equation}
Larger positive values of $E_i$ correspond to observations that are less consistent with the Stage-1 fit and the governing PDE.

The energy scores are computed once from the frozen Stage-1 parameters $(\theta^0,\Phi^0)$ and are then held fixed during Stage~2. Keeping them fixed decouples weight estimation from parameter updates and avoids a moving-target objective.

\subsection{Energy-weighted flow matching}
\label{sec:erfm-weighted-loss}

Stage~2 reweights the observation-conditioned flow-matching loss according to the energy scores. For each observation $(\boldsymbol{\xi}_i,y_i)$, we assign the weight
\[
w_i := w(E_i),
\]
where $w:\mathbb{R}\to(0,1]$ is decreasing. In implementation, we use
\begin{equation}
w(E)=\sigma(-\lambda E),
\qquad
\sigma(r)=\frac{1}{1+\exp(-r)},
\label{eq:weighting-function}
\end{equation}
with inverse-temperature parameter $\lambda>0$. Thus, observations with small energy receive weights close to one, while observations with large energy are downweighted.

Using these weights, we define the ERFM objective
\begin{equation}
\mathcal{L}_{\mathrm{ERFM}}(\theta,\Phi)
:=
\mathbb{E}_{\substack{
t\sim\mathrm{Unif}[0,1]\\
(\boldsymbol{\xi}_i,y_i)\sim p_{\mathrm{data}}\\
\epsilon\sim p_0
}}
\Bigl[
w_i
\bigl\|
v_\theta\bigl(
t,
\bar{s}(t,\boldsymbol{\xi}_i),
\boldsymbol{\xi}_i,
a_\Phi(\boldsymbol{\xi}_i)
\bigr)
-
(y_i-\epsilon)
\bigr\|_2^2
\Bigr].
\label{eq:er-fm-loss}
\end{equation}

This reweighting does not discard observations. Instead, it changes how strongly each sample contributes to the transport update, reducing the influence of observations assigned large energy by the frozen teacher.

\subsection{Reweighting interpretation}
\label{sec:erfm-theory}

The Stage-2 objective in \eqref{eq:er-fm-loss} can be interpreted as standard observation-conditioned flow matching under a reweighted data distribution. The energy weights do not change the bridge construction or the transport dynamics themselves; instead, they alter the effective sampling measure over the observations during fine-tuning.

To make this precise, define the normalized weights
\begin{equation}
\tilde w_i
:=
\frac{w_i}{\sum_{j=1}^N w_j},
\qquad
i=1,\dots,N,
\label{eq:normalized-erfm-weights}
\end{equation}
and let
\begin{equation}
\tilde p(i):=\tilde w_i
\label{eq:reweighted-erfm-distribution}
\end{equation}
denote the corresponding empirical distribution over the observation indices.

Using the same bridge construction in Section~\ref{sec:cfm-obs},
\[
\bar{s}(t,\boldsymbol{\xi}_i)
=
(1-t)\epsilon+t y_i,
\qquad
\epsilon\sim p_0,
\quad
t\sim\mathrm{Unif}[0,1],
\]
consider the standard observation-conditioned flow-matching objective obtained by sampling the observation index from the weighted empirical distribution $\tilde p$:
\begin{equation}
\widetilde{\mathcal{L}}_{\mathrm{CFM}}(\theta,\Phi)
:=
\mathbb{E}_{\substack{
t\sim\mathrm{Unif}[0,1]\\
i\sim \tilde p\\
\epsilon\sim p_0
}}
\Bigl[
\bigl\|
v_\theta\bigl(
t,
\bar{s}(t,\boldsymbol{\xi}_i),
\boldsymbol{\xi}_i,
a_\Phi(\boldsymbol{\xi}_i)
\bigr)
-
(y_i-\epsilon)
\bigr\|_2^2
\Bigr].
\label{eq:reweighted-standard-cfm}
\end{equation}

The following result shows that the ERFM objective is equivalent, up to a positive constant factor, to standard observation-conditioned flow matching under the tilted empirical distribution $\tilde p$.

\begin{theorem}
\label{thm:erfm-reweight}

Let the energy scores $\{E_i\}_{i=1}^N$ be fixed during Stage~2, and let $w_i=w(E_i)>0$ denote the corresponding weights. Then the ERFM objective in \eqref{eq:er-fm-loss} satisfies
\begin{equation}
\mathcal{L}_{\mathrm{ERFM}}(\theta,\Phi)
=
C_w\,
\widetilde{\mathcal{L}}_{\mathrm{CFM}}(\theta,\Phi),
\label{eq:erfm-reweight-identity}
\end{equation}
where
\begin{equation}
C_w
:=
\frac{1}{N}\sum_{j=1}^N w_j
>0.
\label{eq:erfm-normalization-constant}
\end{equation}
Consequently,
\begin{equation}
\nabla_{(\theta,\Phi)}
\mathcal{L}_{\mathrm{ERFM}}(\theta,\Phi)
=
C_w\,
\nabla_{(\theta,\Phi)}
\widetilde{\mathcal{L}}_{\mathrm{CFM}}(\theta,\Phi),
\label{eq:erfm-gradient-identity}
\end{equation}
and the two objectives have the same minimizers.
\end{theorem}

\begin{proof}
For each observation index $i$, define
\[
\ell_i(\theta,\Phi)
:=
\mathbb{E}_{\substack{
t\sim\mathrm{Unif}[0,1]\\
\epsilon\sim p_0
}}
\Bigl[
\bigl\|
v_\theta\bigl(
t,
\bar{s}(t,\boldsymbol{\xi}_i),
\boldsymbol{\xi}_i,
a_\Phi(\boldsymbol{\xi}_i)
\bigr)
-
(y_i-\epsilon)
\bigr\|_2^2
\Bigr].
\]

Using the empirical distribution in \eqref{eq:normalized-erfm-weights}--\eqref{eq:reweighted-erfm-distribution}, the weighted flow-matching objective in \eqref{eq:reweighted-standard-cfm} can be written as
\[
\widetilde{\mathcal{L}}_{\mathrm{CFM}}(\theta,\Phi)
=
\sum_{i=1}^N
\tilde w_i\,\ell_i(\theta,\Phi)
=
\frac{1}{\sum_{j=1}^N w_j}
\sum_{i=1}^N
w_i\,\ell_i(\theta,\Phi).
\]

On the other hand, by the definition of the empirical expectation in \eqref{eq:er-fm-loss},
\[
\mathcal{L}_{\mathrm{ERFM}}(\theta,\Phi)
=
\frac{1}{N}
\sum_{i=1}^N
w_i\,\ell_i(\theta,\Phi).
\]

Combining the two identities yields
\[
\mathcal{L}_{\mathrm{ERFM}}(\theta,\Phi)
=
\left(
\frac{1}{N}\sum_{j=1}^N w_j
\right)
\widetilde{\mathcal{L}}_{\mathrm{CFM}}(\theta,\Phi),
\]
which proves \eqref{eq:erfm-reweight-identity}. Since the energy scores are computed once from the frozen Stage-1 teacher and remain fixed throughout Stage~2, the constant $C_w$ is independent of $(\theta,\Phi)$. Differentiating \eqref{eq:erfm-reweight-identity} therefore yields \eqref{eq:erfm-gradient-identity}. Because $C_w>0$, the two objectives have the same minimizers.
\end{proof}

Theorem~\ref{thm:erfm-reweight} shows that ERFM is equivalent to standard observation-conditioned flow matching under an empirical distribution tilted toward low-energy observations. In this sense, Stage~2 does not modify the transport mechanism itself; rather, it changes the effective training distribution seen by the optimizer. Observations assigned large energy by the frozen Stage-1 teacher contribute less during fine-tuning, whereas observations that are more compatible with the learned physics--data structure receive larger weights.

This interpretation also has a natural population counterpart. Let
$Z=(\boldsymbol{\xi},y)$ denote a generic observation with distribution
$p_{\mathrm{data}}$, and let $E(Z)$ be the frozen energy score induced by the
Stage-1 teacher. The population analogue of the Stage-2 objective is the
weighted risk
\[
\mathbb{E}_{Z\sim p_{\mathrm{data}}}
\bigl[
w(E(Z))\,\ell(\theta,\Phi;Z)
\bigr],
\]
where $\ell(\theta,\Phi;Z)$ denotes the observation-conditioned flow-matching
integrand associated with $Z$. Define the normalization constant
\[
C_w
:=
\int
w(E(z))
\,p_{\mathrm{data}}(dz),
\]
and assume $0<C_w<\infty$. This induces the probability measure
\begin{equation}
q(dz)
:=
\frac{
w(E(z))\,p_{\mathrm{data}}(dz)
}{
C_w
},
\label{eq:population-tilted-distribution}
\end{equation}
which is the energy-tilted version of the original data distribution. Under
this definition,
\[
\mathbb{E}_{Z\sim p_{\mathrm{data}}}
\bigl[
w(E(Z))\,\ell(\theta,\Phi;Z)
\bigr]
=
C_w\,
\mathbb{E}_{Z\sim q}
\bigl[
\ell(\theta,\Phi;Z)
\bigr].
\]

Hence, at the population level, ERFM is equivalent up to the positive constant
$C_w$ to standard observation-conditioned flow matching under the tilted
distribution $q$. In this sense, the second stage concentrates the effective
training distribution on observations that are more compatible with the
Stage-1 physics--data structure.

% This also clarifies the scope of the ``unlearning'' interpretation. ERFM does not provide exact machine unlearning in the sense of removing all information associated with a sample from the trained parameters. Rather, it performs a distributional reweighting step: observations with small weights exert less influence on the Stage-2 optimization trajectory, so high-energy samples are progressively suppressed during fine-tuning.\rm{remove?}

\subsection{Stage-2 refinement}
\label{sec:erfm-refinement}

In practice, the Stage-2 optimization augments the energy-weighted flow-matching term with the same physics regularizers used in Stage~1. The resulting objective is
\begin{equation}
\begin{split}
\mathcal{J}_{\mathrm{ERFM}}(\theta,\Phi)
= {}& \mathcal{L}_{\mathrm{ERFM}}(\theta,\Phi)
      + \lambda_{\mathrm{loc}} \mathcal{L}_{\mathrm{PDE,loc}}(\theta,\Phi) \\
    & + \lambda_{\mathrm{PDE}} \mathcal{L}_{\mathrm{PDE}}(\theta,\Phi)
      + \lambda_{\mathrm{BC}}^u \mathcal{L}_{\mathrm{BC}}^u(\theta,\Phi)
      + \lambda_{\mathrm{BC}}^a \mathcal{L}_{\mathrm{BC}}^a(\Phi),
\end{split}
\label{eq:erfm-loss}
\end{equation}
where $\mathcal{L}_{\mathrm{PDE,loc}}$ and $\mathcal{L}_{\mathrm{PDE}}$ are the local and global physics penalties defined in \eqref{eq:local-pde-loss} and \eqref{eq:pde-loss-pgcfm}, respectively. If boundary or initial-condition penalties are included in Stage~1, the corresponding terms may also be retained in Stage~2.

The coefficients $\lambda_{\mathrm{loc}}$ and $\lambda_{\mathrm{PDE}}$ are typically set equal to their Stage-1 counterparts. Stage~2 is not intended to relearn the full physics-guided model from scratch. Instead, it refines the Stage-1 solution using an energy-weighted observation loss, in which low-energy observations have larger influence on the transport update and high-energy observations are downweighted. The retained physics penalties act as soft constraints that stabilize the transport dynamics and the recovered fields during fine-tuning.

Stage~2 is initialized from the PG-CFM parameters $(\theta^0,\Phi^0)$ and run for fewer optimization steps. Together with the reweighting interpretation in Section~\ref{sec:erfm-theory}, this shows that ERFM may be viewed as a physics-informed refinement stage: it shifts the effective training distribution toward observations that are more compatible with the frozen teacher, while retaining enough physical regularization to prevent the fine-tuned model from drifting away from the PDE structure learned in Stage~1.

Algorithm~\ref{alg:erfm} summarizes one fine-tuning step of ERFM, including the energy-weighted flow-matching term, the retained physics penalties, and the optional boundary or initial-condition contributions.

\begin{algorithm}[t]
\caption{Stage~2: ERFM fine-tuning step (one minibatch)}
\label{alg:erfm}
\begin{algorithmic}[1]
  \STATE \textbf{Input:} observation minibatch $(\boldsymbol{\xi},y,E)$; current parameters $(\theta,\Phi)$ 
  \STATE \# Stage~2 is initialized from $(\theta^0,\Phi^0)$, and energies $E$ are precomputed from the frozen Stage-1 model via \eqref{eq:energy-score}
  \STATE sample $t \sim \mathrm{Unif}[\varepsilon_{\mathrm{fm}},1-\varepsilon_{\mathrm{fm}}]$, $\epsilon \sim p_0$
  \STATE $s_t \leftarrow (1-t)\epsilon + ty$ \hfill \# bridge to noisy observation
  \STATE $u_{\mathrm{loc}} \leftarrow \mathrm{FlowStep}(v_\theta,s_t,\boldsymbol{\xi},a_\Phi)$
  \STATE  $w(E)=\sigma(-\lambda E)$
  \STATE $\mathcal{L}_{\mathrm{ERFM}} \leftarrow w \cdot \mathrm{FM\_loss}(v_\theta,s_t,y,\epsilon)$ \hfill \# \eqref{eq:er-fm-loss}
  \STATE $\mathcal{L}_{\mathrm{PDE,loc}} \leftarrow \mathrm{LocalPDE\_loss}(u_{\mathrm{loc}},a_\Phi,\boldsymbol{\xi})$
  \STATE $u_{\mathrm{loc}} \leftarrow \Gamma_{t\to1}^{(K_1)}(s_t,\boldsymbol{\xi},a_\Phi(\boldsymbol{\xi});\theta)$
  \STATE $\mathcal{J}_{\mathrm{ERFM}} \leftarrow \mathcal{L}_{\mathrm{ERFM}} + \lambda_{\mathrm{loc}}\mathcal{L}_{\mathrm{PDE,loc}} + \lambda_{\mathrm{PDE}}\mathcal{L}_{\mathrm{PDE}}$ \hfill \# \eqref{eq:erfm-loss}
  \STATE \textbf{if} boundary or initial penalties are retained \textbf{then}
  \STATE \hspace{1.5em} $\mathcal{J}_{\mathrm{ERFM}} \leftarrow \mathcal{J}_{\mathrm{ERFM}} + \lambda_{\mathrm{BC}}^{u}\mathcal{L}_{\mathrm{BC}}^u + \lambda_{\mathrm{BC}}^{a}\mathcal{L}_{\mathrm{BC}}^a$
  \STATE update $(\theta,\Phi)$ by one optimizer step on $\mathcal{J}_{\mathrm{ERFM}}$
\end{algorithmic}
\end{algorithm}

\section{Numerical Experiments}
\label{sec:experiments}

We evaluate PG-CFM-ERFM on five inverse problems: the Poisson inverse problem (PInv), the Navier--Stokes inverse problem (NSInv), the Heat inverse problem (HInv), the Burgers inverse problem (BInv), and the Wave inverse problem (WInv). These examples cover elliptic, parabolic, and hyperbolic PDEs and include both spatially varying coefficient fields and global scalar parameters. The main quantitative and qualitative comparisons focus on PInv and NSInv, while HInv, BInv, and WInv serve as additional test problems.

All experiments are repeated over three random seeds, and we report mean
$\pm$ standard deviation. Our primary metric is the relative $L^2$ error
(L2RE). Let $\hat{\mathbf{q}}=(\hat q_1,\dots,\hat q_n)$ and
$\mathbf{q}^\star=(q_1^\star,\dots,q_n^\star)$ denote the prediction and the reference solution on the test grid, respectively. We define
\begin{equation}
\mathrm{L2RE}(\hat{\mathbf{q}})
=
\frac{\|\hat{\mathbf{q}}-\mathbf{q}^\star\|_2}{\|\mathbf{q}^\star\|_2}
=
\sqrt{
\frac{\sum_{j=1}^{n}(\hat q_j-q_j^\star)^2}
     {\sum_{j=1}^{n}(q_j^\star)^2}
}.
\label{eq:l2re-def}
\end{equation}
We also report the relative $L^1$ error (L1RE), the mean squared error (MSE),
and the maximum absolute error (MAE):
\begin{align}
\mathrm{L1RE}(\hat{\mathbf{q}})
&=
\frac{\|\hat{\mathbf{q}}-\mathbf{q}^\star\|_1}{\|\mathbf{q}^\star\|_1}
=
\frac{\sum_{j=1}^{n}|\hat q_j-q_j^\star|}
     {\sum_{j=1}^{n}|q_j^\star|},
\label{eq:l1re-def}\\
\mathrm{MSE}(\hat{\mathbf{q}})
&=
\frac{1}{n}\sum_{j=1}^{n}(\hat q_j-q_j^\star)^2,
\label{eq:mse-def}\\
\mathrm{MAE}(\hat{\mathbf{q}})
&=
\max_{1\le j\le n}|\hat q_j-q_j^\star|.
\label{eq:mae-def}
\end{align}
%\rm{For field reconstruction problems, these quantities are computed on the corresponding test grid. For scalar-parameter estimation problems such as BInv and WInv, the same formulas reduce to their scalar counterparts.}

The observation data are obtained by evaluating the reference solution on a fixed grid or randomly sampled points and then adding noise. Interior collocation points are used for PDE residual enforcement. When boundary or initial constraints are imposed, additional samples are drawn on the corresponding constraint sets.

\subsection{Experimental setup}
\label{sec:exp-setup}

\subsubsection{Benchmark tasks and data generation}
\label{subsubsec:benchmarks}

We summarize the benchmark inverse problems and their data-generation procedures below.

\paragraph{Poisson inverse problem (PInv)}

We reconstruct the spatially varying diffusion coefficient $a(x,y)$ in the elliptic PDE
\begin{equation}
-\nabla\cdot\bigl(a(x,y)\nabla u(x,y)\bigr) = f(x,y),\qquad (x,y)\in\Omega=[0,1]^2.
\end{equation}
The ground-truth solution is
\begin{equation}
u(x,y)=\sin(\pi x)\sin(\pi y),
\end{equation}
and the source term is chosen so that $(a,u)$ satisfies the PDE with the ground-truth coefficient field
\begin{equation}
a(x,y)=\frac{1}{1+x^2+y^2+(x-1)^2+(y-1)^2},\qquad (x,y)\in\Omega,
\end{equation}
namely
\begin{equation}
\begin{split}
f(x,y) = {}& \frac{2\pi^2 \sin(\pi x)\sin(\pi y)}{1+x^2+y^2+(x-1)^2+(y-1)^2} \\
          & + \frac{2\pi \bigl((2x-1)\cos(\pi x)\sin(\pi y)+(2y-1)\sin(\pi x)\cos(\pi y)\bigr)}
                 {\bigl(1+x^2+y^2+(x-1)^2+(y-1)^2\bigr)^2}.
\end{split}
\end{equation}
To ensure identifiability, we impose the true coefficient values on
\(\partial\Omega\) as a boundary condition for \(a\).

\textbf{Training data.}
We collect $N_D=2{,}500$ observations of $u$ on a uniform $50\times50$ grid in $\Omega$. We corrupt $60\%$ of the grid points with i.i.d.\ Gaussian noise $\mathcal{N}(0,1.0)$ and the remaining $40\%$ with $\mathcal{N}(0,0.01)$. We sample $N_{R}=8{,}192$ interior collocation points and $N_B=2{,}048$ boundary points.

\textbf{Test data.}
We evaluate on a noise-free test set of $N_t=10{,}000$ points on a uniform $100\times100$ grid.

\paragraph{Navier--Stokes inverse problem (NSInv)}

We consider the two-dimensional incompressible Navier--Stokes system
\begin{equation}
\begin{cases}
u_t + \beta_1\bigl(u u_x + v u_y\bigr) = -p_x + \beta_2\bigl(u_{xx}+u_{yy}\bigr),\\
v_t + \beta_1\bigl(u v_x + v v_y\bigr) = -p_y + \beta_2\bigl(v_{xx}+v_{yy}\bigr),\\
u_x + v_y = 0,
\end{cases}
\end{equation}
on the space--time domain $[1,8]\times[-2,2]\times[0,7]$. We use the cylinder-flow benchmark from \citet{raissi2019physics}, for which the ground-truth coefficients are $\beta_1=1$ and $\beta_2=0.01$. The inverse task is to recover $(\beta_1,\beta_2)$ and the pressure field $p(x,y,t)$ from velocity observations.

\textbf{Training data.}
From the benchmark reference solution, we randomly select $N_D=5{,}000$ interior space--time locations
\(
\{(x_D^k,y_D^k,t_D^k)\}_{k=1}^{N_D},
\)
together with the corresponding velocity values
\(
\{(u(x_D^k,y_D^k,t_D^k),\,v(x_D^k,y_D^k,t_D^k))\}_{k=1}^{N_D}.
\)
Among these labeled velocity samples, $60\%$ are corrupted by additive Gaussian noise $\mathcal{N}(0,0.5)$ and the remaining $40\%$ by $\mathcal{N}(0,0.01)$. In addition, we sample $N_{R}=2{,}000$ interior collocation points
\(
\{(x_R^k,y_R^k,t_R^k)\}_{k=1}^{N_R}
\)
for PDE residual enforcement.

\textbf{Test data.}
Testing is performed on the remaining noise-free benchmark velocity samples not used for training.

\paragraph{Heat inverse problem (HInv)}

We reconstruct a spatially varying diffusion coefficient $a(x,y)$ in the parabolic PDE
\begin{equation}
u_t(x,y,t) - \nabla\cdot\bigl(a(x,y)\nabla u(x,y,t)\bigr) = f(x,y,t),
\qquad (x,y,t)\in[-1,1]^2\times[0,1].
\end{equation}
The ground-truth solution is
\begin{equation}
u(x,y,t)=e^{-t}\sin(\pi x)\sin(\pi y),
\end{equation}
and the ground-truth coefficient is
\begin{equation}
a(x,y)=2+\sin(\pi x)\sin(\pi y),\qquad (x,y)\in[-1,1]^2,
\end{equation}
with boundary condition $a(x,y)=2$ on $\partial\Omega$. The forcing term $f$ is defined so that the pair $(a,u)$ satisfies the PDE:
\begin{equation}
\begin{split}
f(x,y,t) = {}& \biggl[ (4\pi^2-1)\sin(\pi x)\sin(\pi y) \\
            & + \pi^2\Bigl( 2\sin^2(\pi x)\sin^2(\pi y)
                    - \cos^2(\pi x)\sin^2(\pi y)
                    - \sin^2(\pi x)\cos^2(\pi y) \Bigr) \biggr] e^{-t}.
\end{split}
\end{equation}

\textbf{Training data.}
We sample $N_D=2{,}500$ observation points uniformly at random in $[-1,1]^2\times[0,1]$ and corrupt $60\%$ with $\mathcal{N}(0,1.0)$ and $40\%$ with $\mathcal{N}(0,0.01)$. We use $N_{R}=8{,}192$ residual points and $N_B=2{,}048$ boundary/initial points.

\textbf{Test data.}
We evaluate on a noise-free uniform grid of size $60\times60\times50$ over $[-1,1]^2\times[0,1]$, i.e., $N_t=180{,}000$ points.

\paragraph{Burgers inverse problem (BInv)}

We consider the one-dimensional viscous Burgers equation on $(x,t)\in[-1,1]\times[0,1]$,
\begin{equation}
u_t + u\,u_x = \nu\,u_{xx},
\qquad (x,t)\in[-1,1]\times[0,1],
\label{eq:burgers}
\end{equation}
with unknown viscosity $\nu$. We impose the initial condition
\begin{equation}
u(x,0)=0.5-0.5\tanh(2.5x),
\end{equation}
and Dirichlet boundary conditions
\begin{equation}
u(-1,t)=0.5-0.5\tanh(2.5(-1-0.5t)),
u(1,t)=0.5-0.5\tanh(2.5(1-0.5t)).
\end{equation}
For $\nu=0.1$, the analytic solution is
\begin{equation}
u(x,t)=0.5-0.5\tanh(2.5(x-0.5t)).
\end{equation}

\textbf{Training data.}
We generate solution values on a uniform $200\times200$ grid over $[-1,1]\times[0,1]$ and randomly select $N_D=10{,}000$ points as labeled observations. We corrupt $60\%$ of these points with additive Gaussian noise $\mathcal{N}(0,1.0)$ and the remaining $40\%$ with $\mathcal{N}(0,0.01)$. We additionally sample $N_R=2{,}000$ interior collocation points, $N_I=100$ initial-condition points on $\{t=0\}$, and $N_B=200$ boundary points on $\{x=-1,1\}$.

\textbf{Test data.}
We use the remaining grid points as a noise-free test set.

\paragraph{One-dimensional wave inverse problem (WInv)}

We consider the 1D wave equation on $(x,t)\in[0,1]\times[0,1]$ with unknown wave speed $c$:
\begin{equation}
u_{tt}(x,t) - c^2 u_{xx}(x,t) = 0.
\label{eq:WInv}
\end{equation}
The initial and boundary conditions are
\begin{equation*}
u(x,0)=\sin(\pi x)+\tfrac{1}{2}\sin(4\pi x),\qquad
u_t(x,0)=0,\qquad
u(0,t)=u(1,t)=0.
\end{equation*}
For $c=2$, the analytic solution is
\begin{equation*}
u(x,t)=\sin(\pi x)\cos(c\pi t)+\tfrac{1}{2}\sin(4\pi x)\cos(4c\pi t).
\end{equation*}

\textbf{Training data.}
We generate data on a uniform $200\times200$ grid over $[0,1]\times[0,1]$ and randomly select $N_D=10{,}000$ points for training, corrupting $60\%$ with $\mathcal{N}(0,1.0)$ and $40\%$ with $\mathcal{N}(0,0.01)$. We sample $N_R=2{,}000$ collocation points in the interior, $N_I=100$ initial points, and $N_B=200$ boundary points.

\textbf{Test data.}
We use the remaining grid points as a noise-free test set.

\subsubsection{Ablation settings}
\label{subsubsec:exp-ablation-settings}

For the physics-placement ablation, we compare CFM
$(\lambda_{\mathrm{loc}}=\lambda_{\mathrm{PDE}}=0)$, PG-CFM with local loss only (CFM-$\ell$, $\lambda_{\mathrm{loc}}>0,\lambda_{\mathrm{PDE}}=0$), and PG-CFM with global
loss only (CFM-g, $\lambda_{\mathrm{loc}}=0,\lambda_{\mathrm{PDE}}>0$), and PG-CFM with both
physics losses active. This isolates the respective contributions of the
bridge-local and global physics terms.

% For the second-stage reweighting ablation, all methods first train the
% Stage-1 PG-CFM model. The second-stage alternatives are continued PG-CFM
% training with the same unweighted objective, self-distilled PG-CFM retrained
% on labels obtained from the Stage-1 predictions, self-refined PG-CFM
% initialized from the Stage-1 parameters and further trained on the same
% refined labels, and ERFM applied to the original noisy observations with
% energy-based weights. These controls separate the effect of energy
% reweighting from that of additional optimization and label refinement.

For the second-stage reweighting ablation, all methods start from the same trained Stage-1 PG-CFM model and use the same Stage-2 optimization budget. We compare ERFM with continued PG-CFM training and label-refinement controls, as defined in the baseline descriptions below. This comparison is used to distinguish the effect of energy-based reweighting from that of additional optimization or replacement of the noisy labels by Stage-1 predictions.

For the robustness ablations, we use NSInv as the representative benchmark and
examine three aspects of the corruption model: the corruption ratio $\rho$,
the corrupted-noise scale $\sigma_{\mathrm{bad}}$, and the noise type. To study
the effect of the corruption ratio, we vary
$\rho\in\{0.2,0.4,0.6,0.8,1.0\}$ while fixing
$\sigma_{\mathrm{bad}}=0.5$. To study the effect of the noise scale, we vary
$\sigma_{\mathrm{bad}}\in\{0.1,0.3,0.5,0.7,0.9\}$ while fixing $\rho=0.6$.
To study the effect of the noise type, we consider Gaussian, Laplacian,
Student-$t$, uniform, and outlier-mixture noise. For the Laplacian distribution,
the scale is set to $b=\sigma/\sqrt{2}$; for Student-$t$, we use $\nu=3$; and
for uniform noise, we use $\mathrm{Unif}(-\sqrt{3}\sigma,\sqrt{3}\sigma)$ so
that the variance is $\sigma^2$. The outlier mixture is defined by drawing
from $\mathcal{N}(0,2.0^2)$ with probability $0.1$ and from
$\mathcal{N}(0,\sigma^2)$ otherwise.

\subsubsection{Baselines}
\label{subsubsec:baselines}

Unless otherwise stated, all baselines are given the same observation set, corruption protocol, and available collocation or boundary/initial-condition data. For the grid-based generative baselines, the same observations are first interpolated onto regular grids required by their architectures.

\paragraph{PINN}
The vanilla PINN baseline \citep{raissi2019physics} learns a direct
coordinate-to-state map together with the unknown coefficient field (or scalar
parameters) by minimizing a composite objective consisting of a data-misfit
term, a PDE residual term, and boundary or initial-condition penalties.

\paragraph{R-PINN}
R-PINN uses the same architecture and training data as PINN, but replaces the
quadratic data-misfit term by a robust Huber-type penalty
\citep{meyer2021alternative}. This isolates the effect of loss-level
robustification within the PINN framework.

\paragraph{B-PINN}
B-PINN \citep{yang2021b} uses the same PINN formulation but treats the network
parameters in a Bayesian manner and forms predictions by posterior averaging.
It serves as a probabilistic PINN baseline under the same data and physics
constraints.

\paragraph{CFM}
The CFM baseline uses the same observation-conditioned bridge construction as
our method, but is trained only with the flow-matching loss
$\mathcal{L}_{\mathrm{CFM}}$ and does not include PDE regularization. This
isolates the contribution of the physics terms.

\paragraph{PG-CFM}
PG-CFM is the Stage-1 version of our method. It uses the same coordinate-based
flow model and observations as the full method, but is trained only with the
physics-guided objective in \eqref{eq:pg-cfm-loss}, without the Stage-2
energy-based refinement.

\paragraph{SD-PG-CFM and SF-PG-CFM}
SD-PG-CFM (\emph{self-distilled PG-CFM}) replaces the noisy observations with
labels obtained from the trained Stage-1 PG-CFM predictions and retrains
PG-CFM on these refined labels. SF-PG-CFM (\emph{self-refined PG-CFM}) uses
the same refined labels but continues optimization from the Stage-1 parameters
rather than retraining from scratch. These controls separate the effect of
second-stage reweighting from that of label refinement alone.

\paragraph{Continued PG-CFM}
We also consider a continued-training version of PG-CFM with the same Stage-2
optimization budget but without energy reweighting, that is, with uniform
sample weights. This isolates the effect of reweighting from that of
additional optimization alone.

\paragraph{PIDM and PBFM}
For the main PInv and NSInv benchmarks, we additionally compare with the
grid-based generative physics baselines PIDM (including the two variants
PIDM-ME and PIDM-SE) \citep{bastek2025physics} and PBFM
\citep{baldan2025flow}. These methods are trained on the same inverse
instance and use the same noisy observations, but unlike our coordinate-based
formulation they require the sparse observations to be interpolated onto
regular grids before training. This comparison shows how grid-dependent
generative baselines perform in the same corrupted inverse setting.

\subsubsection{Implementation details and training settings}
\label{subsubsec:implementation}

Our model consists of a time-dependent vector field $v_\theta$ and either a
coefficient network $a_\Phi$ or a trainable scalar parameter, depending on the
benchmark. For time-dependent problems, we use $t$ for the flow time and
$\tau$ for the physical time variable. The vector field takes as input the
current state, the flow time, the physical coordinates, and problem-dependent
inputs when applicable. The benchmark-specific network parameterizations are
listed in Table~\ref{tab:nn-architectures-main}.

For each training pair $(\boldsymbol{\xi},y)$, we sample
$\epsilon\sim p_0=\mathcal N(0,I)$ and
$t\sim\mathrm{Unif}[\varepsilon_{\mathrm{fm}},1-\varepsilon_{\mathrm{fm}}]$,
construct the bridge
$s(t,\boldsymbol{\xi})=(1-t)\epsilon+t y$, and use target velocity
$y-\epsilon$. The local and global physics terms are evaluated with fixed-step
Heun integration: the local term integrates from the sampled bridge time $t$
to $1$ using $K_1$ steps, while the global term integrates from $0$ to $1$
using $K_2$ steps and starts from the deterministic base state
$u_{\mathrm{base}}=\mathbf 0$. The global physics term and the boundary or
initial-condition terms are evaluated every 10 optimization
steps. Stage~1 minimizes the PG-CFM objective \eqref{eq:pg-cfm-loss}, and
Stage~2 is initialized from the Stage-1 parameters and minimizes the ERFM
objective \eqref{eq:erfm-loss}. The energy scores are computed once from the
frozen Stage-1 model, normalized by the median and median absolute deviation,
and converted into weights $w_i=\sigma(-\lambda E_i)$. We use Adam in both
stages. The benchmark-specific integration settings and training
hyperparameters are summarized in
Table~\ref{tab:training-config-merged}.

Hyperparameters are chosen in a stage-wise manner. In Stage~1, we tune the
physics weights $(\lambda_{\mathrm{loc}},\lambda_{\mathrm{PDE}})$ to obtain a
stable baseline model. In Stage~2, we use
$\omega_{\mathrm{obs}}=1.0$ and $\omega_{\mathrm{phys}}=0.1$ as default values,
reducing $\omega_{\mathrm{phys}}$ only when the PDE residual term is
comparatively large. We also use $\kappa=0.5$ and $\lambda=5$ as default
settings in most experiments. Across benchmarks, the main differences are the weights of the local and global physics penalties, with smaller adjustments to \(\omega_{\mathrm{obs}}\) and \(\omega_{\mathrm{phys}}\) in the Stage-2 energy score.

\begin{table}[t]
\centering
\caption{Neural parameterizations used in PG-CFM-ERFM. For time-dependent problems, $t$ denotes the flow time and $\tau$ denotes the physical time variable.}
\label{tab:nn-architectures-main}
\small
\begin{tabular}{llll}
\toprule
Benchmark & Network & Mapping & Architecture \\
\midrule
NSInv
& $v_\theta$
& $(u,v,t,x,y,\tau)\mapsto\mathbb R^2$
& 3-layer MLP, width 256, SiLU \\
& $p_\Phi$
& $(x,y,\tau)\mapsto\mathbb R$
& 4-layer MLP, width 256, SiLU \\
& $(\beta_1,\beta_2)$
& trainable scalars
& -- \\
\midrule
PInv
& $v_\theta$
& $(u,t,x,y,f)\mapsto\mathbb R$
& 3-layer MLP, width 64, SiLU \\
& $a_\Phi$
& $(x,y)\mapsto\mathbb R$
& 4-layer MLP, width 64, SiLU \\
\midrule
HInv
& $v_\theta$
& $(u,t,x,y,\tau,a_\Phi,f)\mapsto\mathbb R$
& 6-layer MLP, width 128, Tanh \\
& $a_\Phi$
& $(x,y)\mapsto\mathbb R$
& 4-layer MLP, width 128, Tanh \\
\midrule
WInv
& $v_\theta$
& $(u,t,x,\tau)\mapsto\mathbb R$
& 3-layer MLP, width 256, SiLU \\
& $c$
& trainable scalar
& -- \\
\midrule
BInv
& $v_\theta$
& $(u,t,x,\tau,f)\mapsto\mathbb R$
& 3-layer MLP, width 64, SiLU \\
& $\nu$
& trainable scalar
& -- \\
\bottomrule
\end{tabular}
\end{table}

\begin{table}[t]
\centering
\caption{Benchmark-specific integration settings and training hyperparameters for PG-CFM-ERFM. Here ``lr'' denotes the learning rate.}
\label{tab:training-config-merged}
\small
\setlength{\tabcolsep}{3pt}
\resizebox{\linewidth}{!}{
\begin{tabular}{lccccccccc}
\toprule
Benchmark
& $K_1$
& $K_2$
& Stage 1 lr
& Stage 1 epochs
& Stage 2 lr
& Stage 2 epochs
& $\lambda$
& $\kappa$
& $\omega_{\mathrm{phys}}$ \\
\midrule
NSInv & 5  & 10 & $10^{-3}$        & 5000 & $10^{-3}$        & 1000 & 5.0 & 0.5 & 0.10 \\
PInv  & 5  & 10 & $10^{-3}$        & 7000 & $10^{-4}$        & 1000 & 5.0 & 0.5 & 0.05 \\
HInv  & 10 & 10 & $10^{-3}$        & 3600 & $10^{-3}$        & 500  & 1.0 & 0.0 & 0.10 \\
WInv  & 10 & 10 & $10^{-3}$        & 5000 & $2\times10^{-4}$ & 3000 & 5.0 & 0.5 & 0.10 \\
BInv  & 10 & 10 & $10^{-3}$        & 5000 & $10^{-3}$        & 1000 & 5.0 & 0.5 & 0.01 \\
\bottomrule
\end{tabular}}
\end{table}

\subsection{Main results on core inverse problems}
\label{sec:main-results}

\subsubsection{Quantitative comparison on PInv and NSInv}
\label{subsubsec:main_quant}

Tables~\ref{tab:poisson_metrics} and~\ref{tab:ns_metrics} summarize the quantitative results on PInv and NSInv under mixed corruption, shown as mean
$\pm$ standard deviation over three random seeds.

On PInv, PG-CFM-ERFM attains the lowest L2RE, L1RE, MSE, and MAE, with MSE reduced to $1.04\mathrm{E}{-}4$. CFM performs much worse under corrupted labels, while the PINN-based baselines remain less accurate; among them, R-PINN provides only marginal and non-uniform gains over PINN, whereas B-PINN exhibits the largest mean error and variance. The grid-based generative baselines PIDM and PBFM also perform substantially worse than the proposed method, which reflects the difficulty of handling sparse, off-grid observations after interpolation onto regular grids.

\begin{table}[t]
\centering
\caption{Mean $\pm$ standard deviation of error metrics for PInv coefficient recovery.}
\label{tab:poisson_metrics}
\setlength{\tabcolsep}{2pt}
\scalebox{0.8}{
\begin{tabular}{lccccccccc}
\toprule
Metric
& PG-CFM-ERFM
& PG-CFM
& CFM
& PINN
& R-PINN
& B-PINN
& PBFM
& PIDM-ME
& PIDM-SE \\
\midrule
L2RE
& \best{1.95E-2} & 2.61E-2 & 1.91E-1 & 5.63E-2 & 5.43E-2 & 3.27E-1 & 1.02E+0 & 6.58E-1 & 9.84E-1 \\
& {\scriptsize $\pm$1.62E-2} & {\scriptsize $\pm$2.55E-2} & {\scriptsize $\pm$3.05E-3}
& {\scriptsize $\pm$1.49E-2} & {\scriptsize $\pm$5.84E-3} & {\scriptsize $\pm$1.89E-1}
& {\scriptsize $\pm$3.05E-2} & {\scriptsize $\pm$8.00E-2} & {\scriptsize $\pm$1.88E-2} \\
\midrule
L1RE
& \best{1.55E-2} & 2.13E-2 & 1.57E-1 & 4.56E-2 & 4.19E-2 & 2.85E-1 & 1.00E+0 & 6.03E-1 & 9.84E-1 \\
& {\scriptsize $\pm$1.34E-2} & {\scriptsize $\pm$2.12E-2} & {\scriptsize $\pm$2.53E-3}
& {\scriptsize $\pm$1.46E-2} & {\scriptsize $\pm$5.41E-3} & {\scriptsize $\pm$1.80E-1}
& {\scriptsize $\pm$7.15E-3} & {\scriptsize $\pm$4.43E-2} & {\scriptsize $\pm$1.86E-2} \\
\midrule
MSE
& \best{1.04E-4} & 2.09E-4 & 6.82E-3 & 6.33E-4 & 5.58E-4 & 2.66E-2 & 1.94E-1 & 8.19E-2 & 1.80E-1 \\
& {\scriptsize $\pm$1.43E-4} & {\scriptsize $\pm$3.17E-4} & {\scriptsize $\pm$2.17E-4}
& {\scriptsize $\pm$3.43E-4} & {\scriptsize $\pm$1.17E-4} & {\scriptsize $\pm$2.50E-2}
& {\scriptsize $\pm$1.17E-2} & {\scriptsize $\pm$2.04E-2} & {\scriptsize $\pm$6.85E-3} \\
\midrule
MAE
& \best{1.72E-2} & 2.22E-2 & 1.59E-1 & 5.61E-2 & 5.69E-2 & 2.63E-1 & 8.73E-1 & 2.59E-1 & 4.23E-1 \\
& {\scriptsize $\pm$1.26E-2} & {\scriptsize $\pm$2.01E-2} & {\scriptsize $\pm$2.49E-3}
& {\scriptsize $\pm$8.81E-3} & {\scriptsize $\pm$1.08E-2} & {\scriptsize $\pm$1.05E-1}
& {\scriptsize $\pm$5.75E-1} & {\scriptsize $\pm$1.90E-2} & {\scriptsize $\pm$8.00E-3} \\
\bottomrule
\end{tabular}}
\end{table}

On NSInv, PG-CFM-ERFM reduces L2RE from $3.17\mathrm{E}{+}1$ for CFM to $4.06\mathrm{E}{-}1$ and also outperforms PINN, R-PINN, and B-PINN by a clear margin. Compared with PIDM and PBFM, PG-CFM-ERFM achieves lower mean errors on all reported metrics, with the clearest gains in relative-error and MSE metrics; the MAE gap against PIDM is smaller. The difference between CFM and PG-CFM indicates that adding PDE residual penalties substantially improves this benchmark under the tested corruption setting. The additional reduction from PG-CFM to PG-CFM-ERFM is consistent with the effect of downweighting observations assigned large physics--data energy by the frozen Stage-1 model.
They also highlight the advantage of the mesh-free coordinate-based formulation: PG-CFM-ERFM directly uses irregular observations, whereas PIDM and PBFM depend on grid interpolation and do not include an intrinsic mechanism for sample-adaptive suppression of corrupted measurements.

\begin{table}[t]
\centering
\caption{Mean $\pm$ standard deviation of error metrics for NSInv pressure recovery.}
\label{tab:ns_metrics}
\setlength{\tabcolsep}{2pt}
\scalebox{0.8}{
\begin{tabular}{lccccccccc}
\toprule
Metric
& PG-CFM-ERFM
& PG-CFM
& CFM
& PINN
& R-PINN
& B-PINN
& PBFM
& PIDM-ME
& PIDM-SE \\
\midrule
L2RE
& \best{4.06E-1} & 5.75E-1 & 3.17E+1 & 3.44E+0 & 3.58E+0 & 1.99E+0 & 1.01E+0 & 1.55E+0 & 1.50E+0 \\
& {\scriptsize $\pm$4.75E-2} & {\scriptsize $\pm$7.38E-2} & {\scriptsize $\pm$7.90E+0}
& {\scriptsize $\pm$1.09E+0} & {\scriptsize $\pm$2.24E+0} & {\scriptsize $\pm$1.84E-1}
& {\scriptsize $\pm$4.50E-3} & {\scriptsize $\pm$2.00E-1} & {\scriptsize $\pm$3.82E-1} \\
\midrule
L1RE
& \best{4.26E-1} & 5.98E-1 & 3.33E+1 & 3.80E+0 & 4.86E+0 & 2.26E+0 & 1.03E+0 & 2.01E+0 & 1.93E+0 \\
& {\scriptsize $\pm$6.13E-2} & {\scriptsize $\pm$7.81E-2} & {\scriptsize $\pm$9.27E+0}
& {\scriptsize $\pm$8.43E-1} & {\scriptsize $\pm$3.30E+0} & {\scriptsize $\pm$1.98E-1}
& {\scriptsize $\pm$3.28E-3} & {\scriptsize $\pm$2.52E-1} & {\scriptsize $\pm$5.16E-1} \\
\midrule
MSE
& \best{1.78E-3} & 3.56E-3 & 1.12E+1 & 2.73E-1 & 3.10E-1 & 1.21E+0 & 1.96E-2 & 4.23E-2 & 4.16E-2 \\
& {\scriptsize $\pm$4.05E-4} & {\scriptsize $\pm$9.18E-4} & {\scriptsize $\pm$5.30E+0}
& {\scriptsize $\pm$1.47E-1} & {\scriptsize $\pm$3.19E-1} & {\scriptsize $\pm$2.26E-1}
& {\scriptsize $\pm$1.74E-4} & {\scriptsize $\pm$1.06E-2} & {\scriptsize $\pm$1.87E-2} \\
\midrule
MAE
& \best{1.70E-1} & 2.34E-1 & 1.21E+1 & 3.46E+0 & 1.44E+0 & 3.86E+0 & 5.35E-1 & 1.83E-1 & 1.76E-1 \\
& {\scriptsize $\pm$3.00E-2} & {\scriptsize $\pm$3.56E-2} & {\scriptsize $\pm$2.06E+0}
& {\scriptsize $\pm$1.03E+0} & {\scriptsize $\pm$2.48E-1} & {\scriptsize $\pm$1.82E-1}
& {\scriptsize $\pm$8.70E-3} & {\scriptsize $\pm$2.30E-2} & {\scriptsize $\pm$4.70E-2} \\
\bottomrule
\end{tabular}}
\end{table}

Table~\ref{tab:nsinv_beta_mse} reports the recovery errors for the NSInv
coefficients $(\beta_1,\beta_2)$. Relative to CFM, both PG-CFM and
PG-CFM-ERFM substantially reduce the coefficient errors. PG-CFM-ERFM gives
the lowest MSE for $\beta_2$ and improves the recovery of $\beta_1$ over most
baselines, although R-PINN gives the smallest MSE for $\beta_1$. Taken
together with Table~\ref{tab:ns_metrics}, these results show that the proposed
method remains competitive for both pressure recovery and coefficient identification on this problem.

\begin{table}[t]
    \centering
    \caption{Mean $\pm$ standard deviation of MSE for the recovered NSInv coefficients $\beta_1$ and $\beta_2$ (true values: $\beta_1=1.0$, $\beta_2=0.01$).}
    \label{tab:nsinv_beta_mse}
    \setlength{\tabcolsep}{2pt}
    \resizebox{\linewidth}{!}{
    \begin{tabular}{l|ccccccccc}
        \toprule
        Methods
        & \textbf{PG-CFM-ERFM}
        & PG-CFM
        & CFM
        & PINN
        & R-PINN
        & B-PINN 
        & PBFM 
        & PIDM-ME
        & PIDM-SE
        \\
        \midrule
        $\beta_1$
        & {1.27E-2} & 6.47E-2 & 2.57E-1 & 2.47E-1 & \best{3.11E-3} & 5.30E-2 & 1.00E+0 & 2.30E-1 & 9.88E-1\\
        & {\scriptsize $\pm$1.36E-3} & {\scriptsize $\pm$1.80E-3} & {\scriptsize $\pm$1.50E-2} & {\scriptsize $\pm$8.49E-2} & {\scriptsize $\pm$1.69E-3} & {\scriptsize $\pm$2.30E-2} & {\scriptsize $\pm$4.07E-4} & {\scriptsize $\pm$3.52E-2} & {\scriptsize $\pm$3.82E-3} \\
        \midrule
        $\beta_2$
        & \best{3.53E-6} & 9.00E-6 & 1.10E-4 & 5.02E-4 & 2.79E-5 & 8.85E-5 & 1.00E-4 & 1.00E-4 & 1.00E-4\\
        & {\scriptsize $\pm$4.70E-6} & {\scriptsize $\pm$1.35E-5} & {\scriptsize $\pm$2.56E-5} & {\scriptsize $\pm$4.54E-4} & {\scriptsize $\pm$6.68E-6} & {\scriptsize $\pm$1.38E-5} & {\scriptsize $\pm$1.10E-6} &
        {\scriptsize $\pm$6.00E-8} &
        {\scriptsize $\pm$9.68E-8}
        \\
        \bottomrule
    \end{tabular}}
\end{table}

\subsubsection{Qualitative comparison}
\label{subsubsec:main_qual}

Figures~\ref{fig:pinv-qualitative} and~\ref{fig:nsinv-qualitative} show
representative reconstructions for PInv and NSInv. On both problems, CFM
produces visible structural error under mixed corruption. PG-CFM recovers the
overall structure more accurately, and PG-CFM-ERFM further reduces localized
artifacts, giving reconstructions that are visually closer to the reference
solution. The PINN-based baselines remain more sensitive to corrupted
observations. On both benchmarks, PINN shows visible bias, R-PINN reduces some artifacts, and B-PINN still exhibits noticeable artifacts. These
visual trends agree with the quantitative results in
Tables~\ref{tab:poisson_metrics} and~\ref{tab:ns_metrics}.

\begin{figure}[t]
\centering
\figpanel{0.235\linewidth}{0.16\textheight}{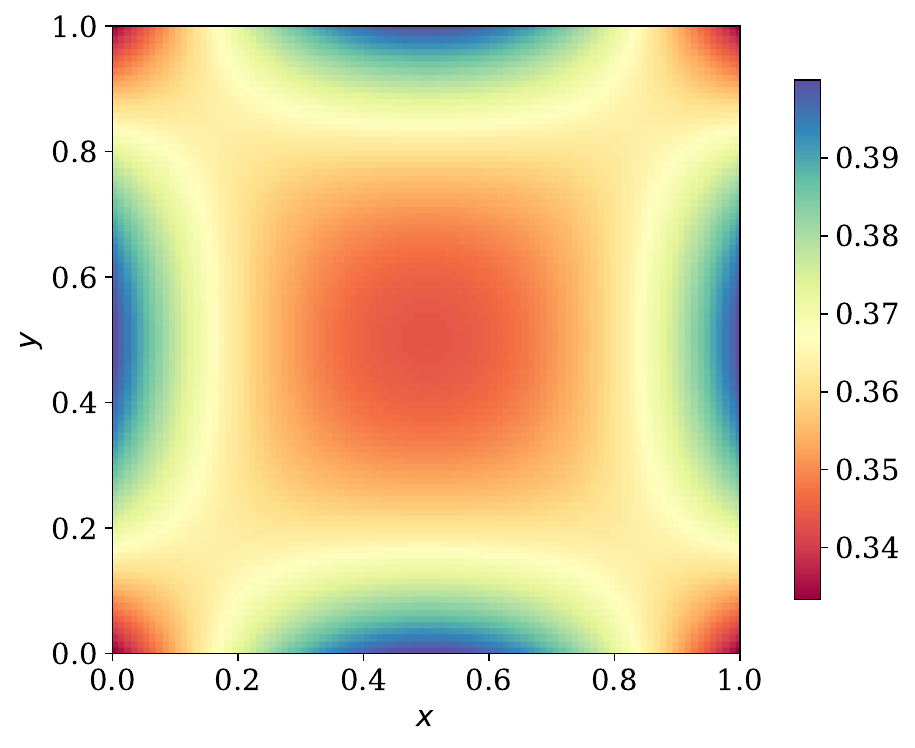}{(a) CFM}
\hfill
\figpanel{0.235\linewidth}{0.16\textheight}{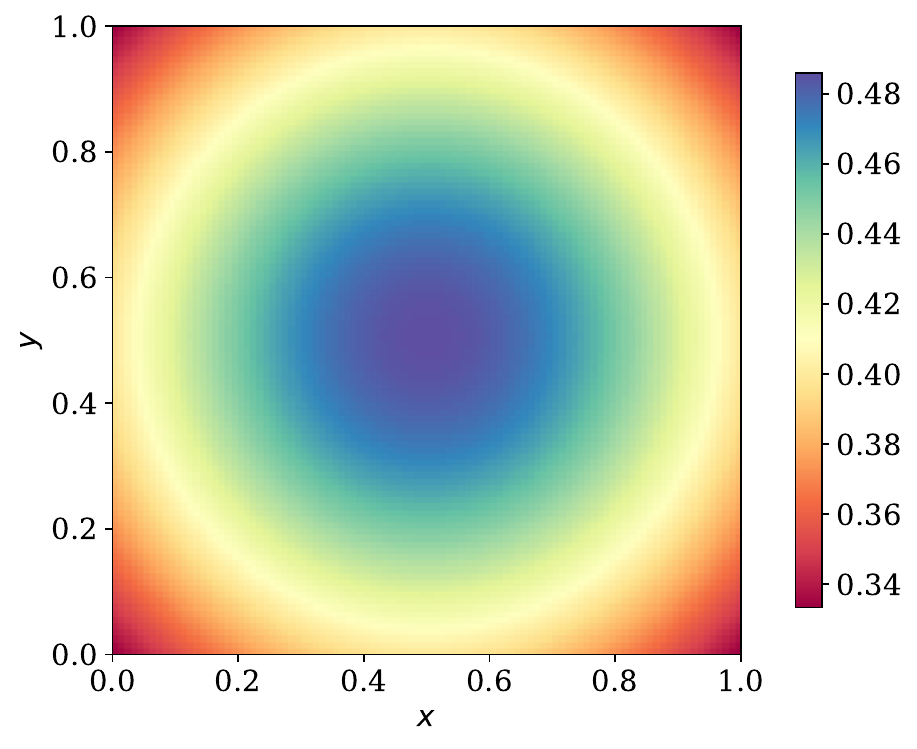}{(b) PG-CFM}
\hfill
\figpanel{0.235\linewidth}{0.16\textheight}{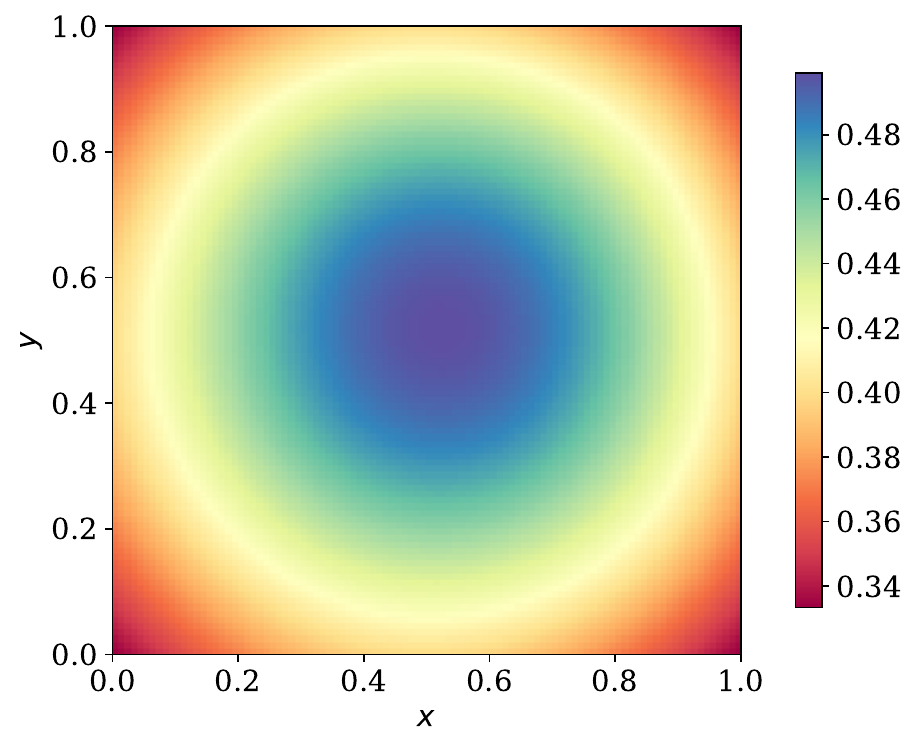}{(c) PG-CFM-ERFM}
\hfill
\figpanel{0.235\linewidth}{0.16\textheight}{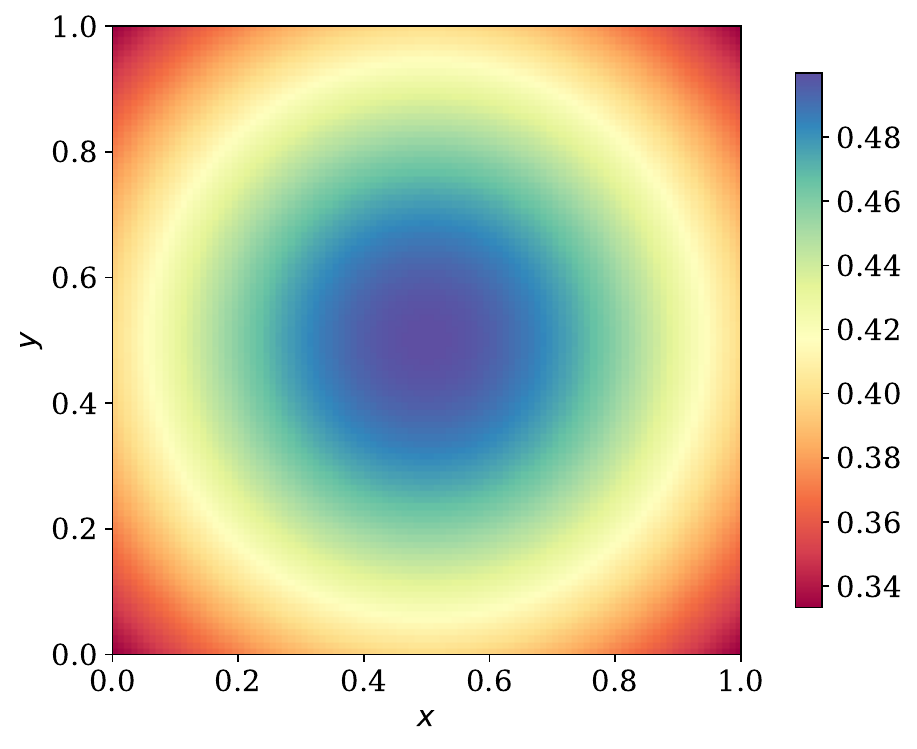}{(d) Ground truth}

\smallskip

\figpanel{0.235\linewidth}{0.16\textheight}{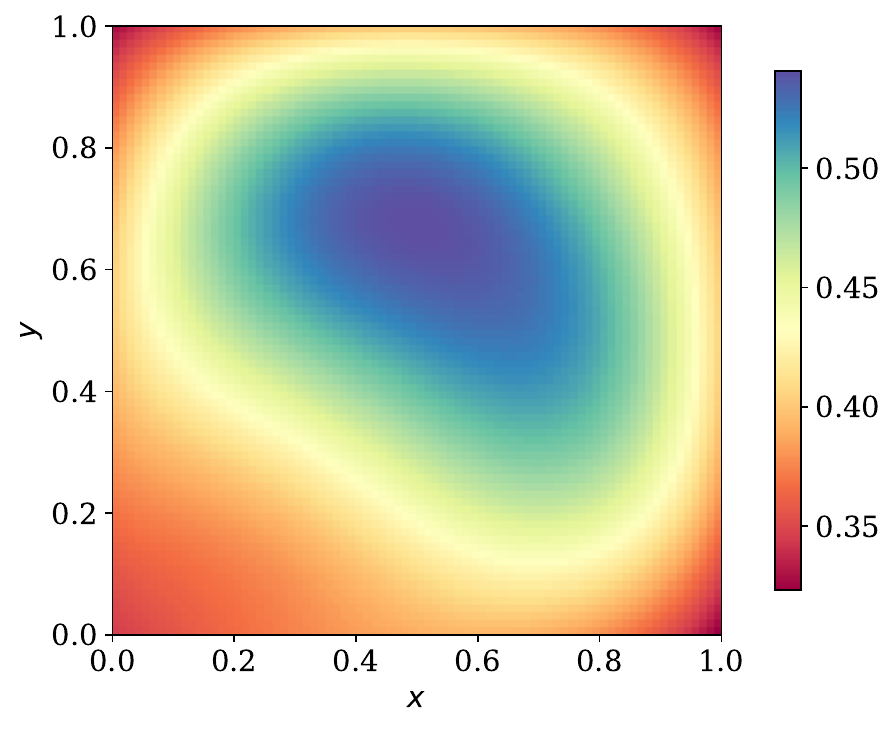}{(e) PINN}
\hfill
\figpanel{0.235\linewidth}{0.16\textheight}{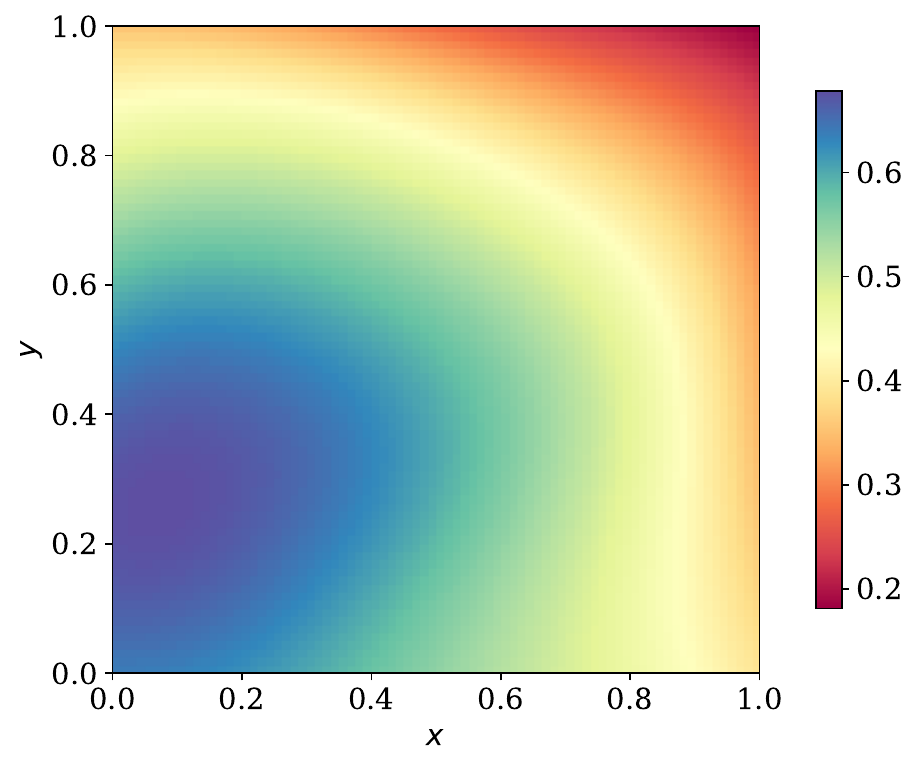}{(f) B-PINN}
\hfill
\figpanel{0.235\linewidth}{0.16\textheight}{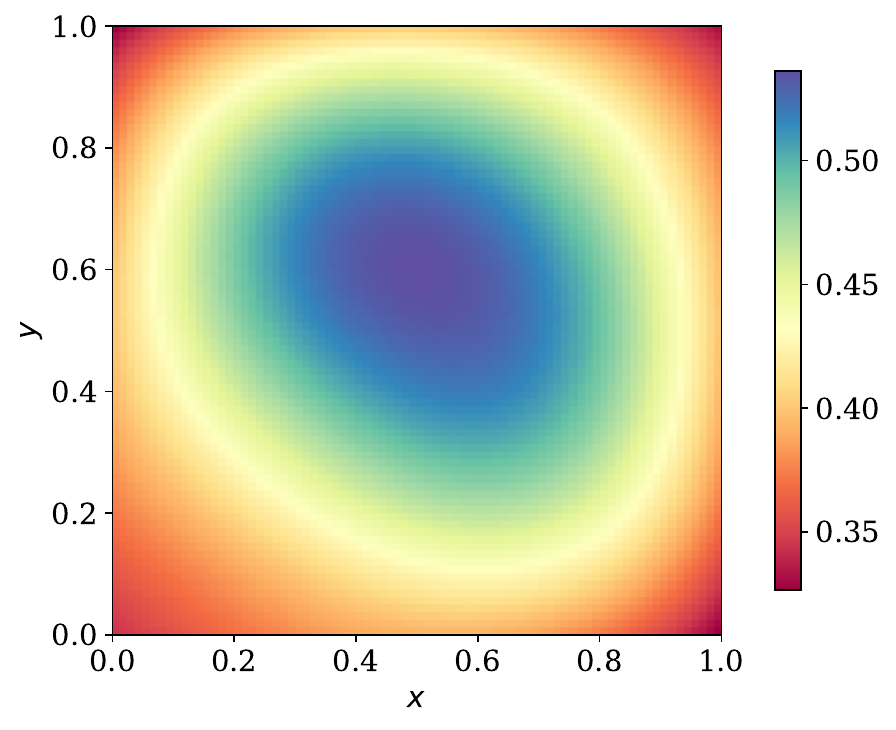}{(g) R-PINN}
\hfill
\figpanel{0.235\linewidth}{0.16\textheight}{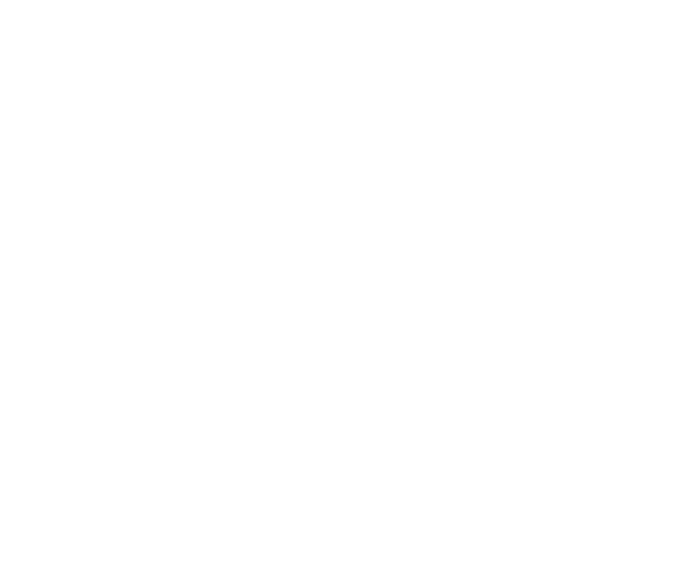}{\phantom{(h)}}

\caption{Representative reconstructions for the PInv coefficient field $a$. Panels (a)--(c) show the results of CFM, PG-CFM, and PG-CFM-ERFM, respectively; panel (d) shows the ground truth; panels (e)--(g) show PINN, B-PINN, and R-PINN.}
\label{fig:pinv-qualitative}
\end{figure}

\begin{figure}[t]
\centering
\figpanel{0.235\linewidth}{0.16\textheight}{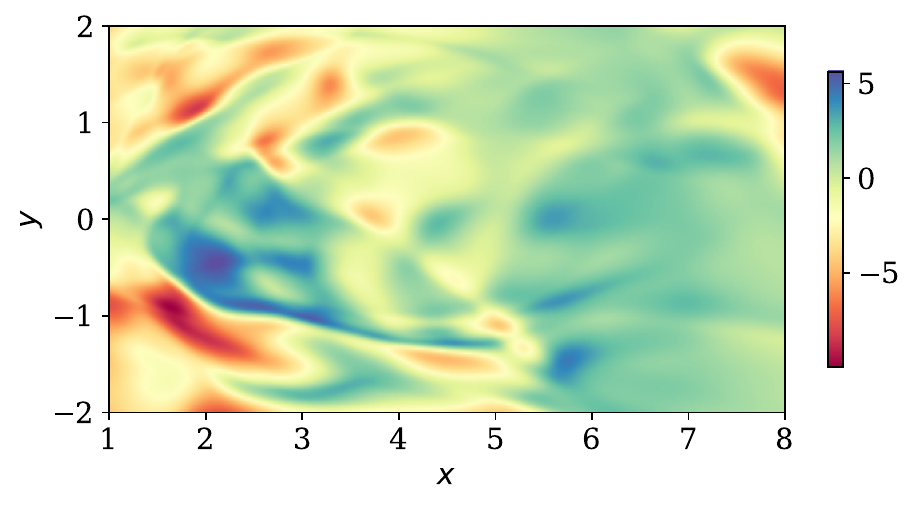}{(a) CFM}
\hfill
\figpanel{0.235\linewidth}{0.16\textheight}{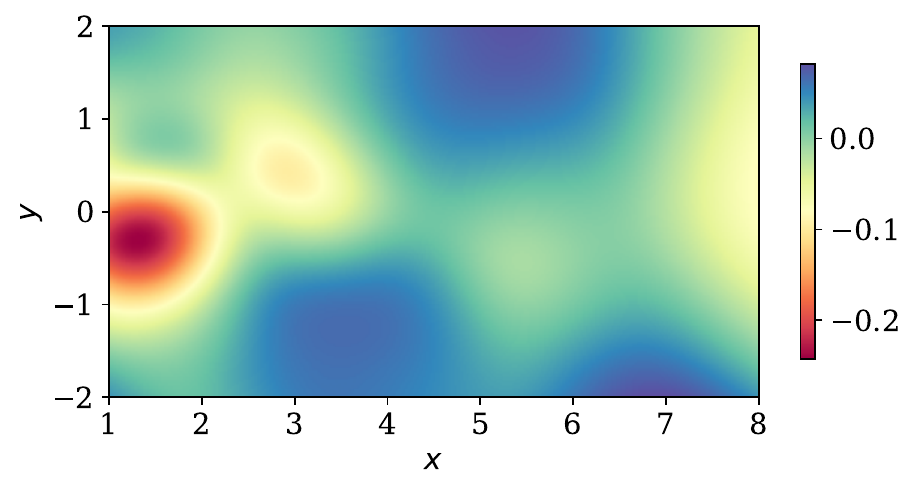}{(b) PG-CFM}
\hfill
\figpanel{0.235\linewidth}{0.16\textheight}{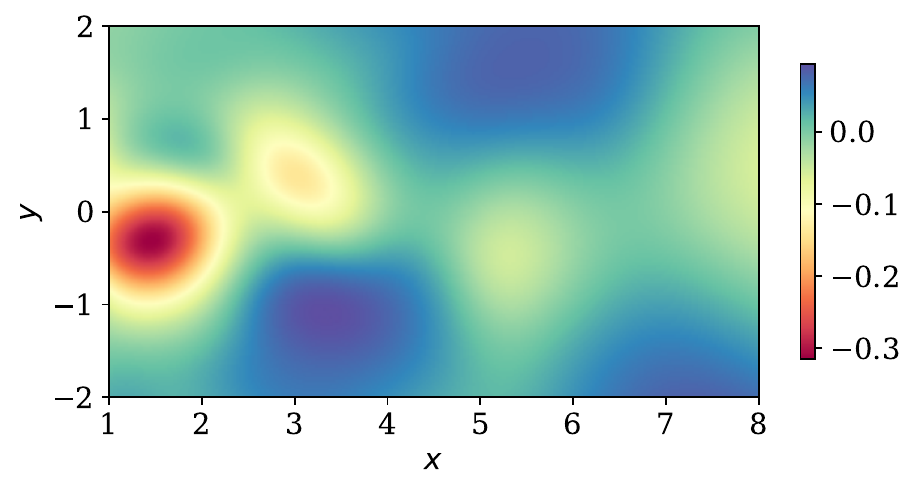}{(c) PG-CFM-ERFM}
\hfill
\figpanel{0.235\linewidth}{0.16\textheight}{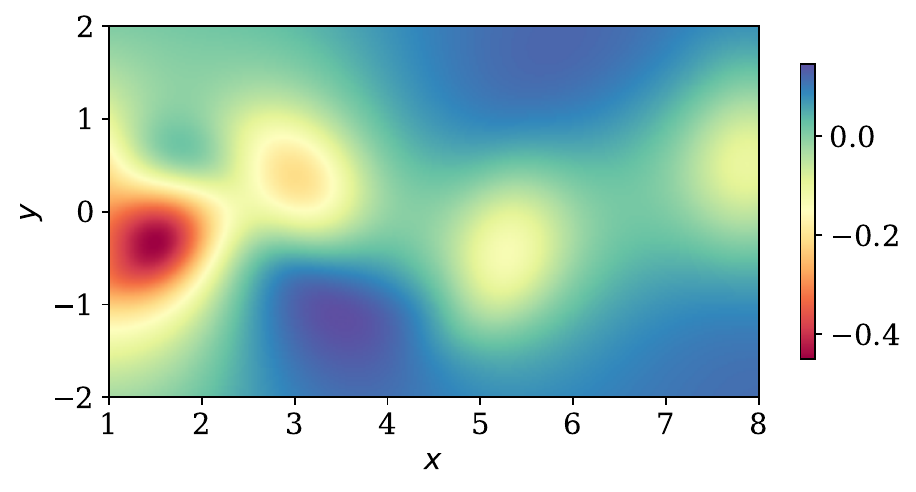}{(d) Ground truth}

\smallskip

\figpanel{0.235\linewidth}{0.16\textheight}{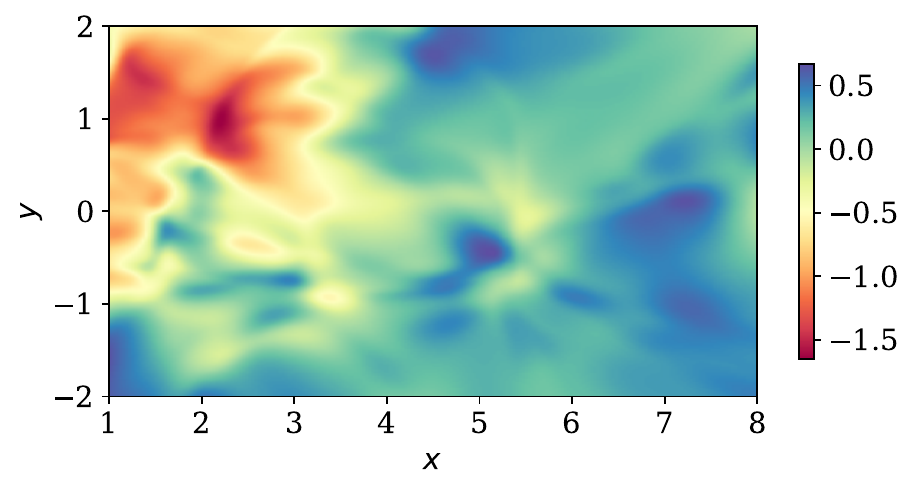}{(e) PINN}
\hfill
\figpanel{0.235\linewidth}{0.16\textheight}{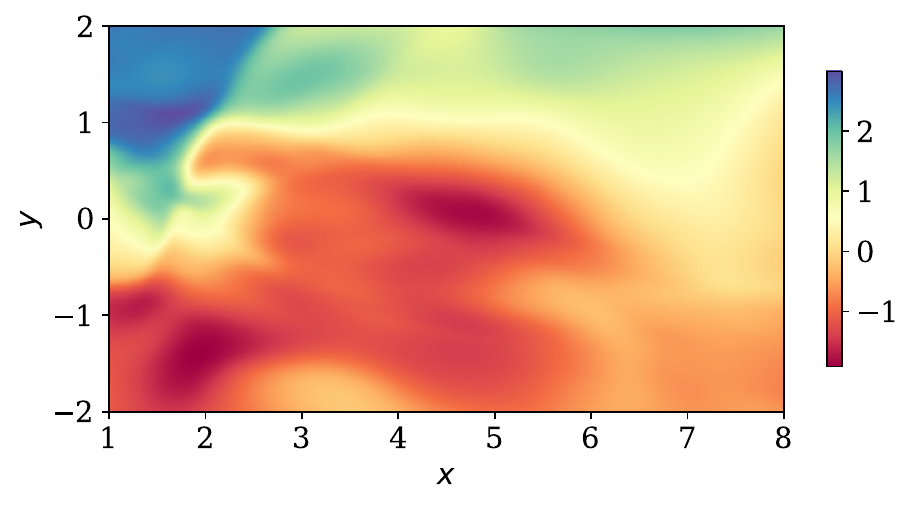}{(f) B-PINN}
\hfill
\figpanel{0.235\linewidth}{0.16\textheight}{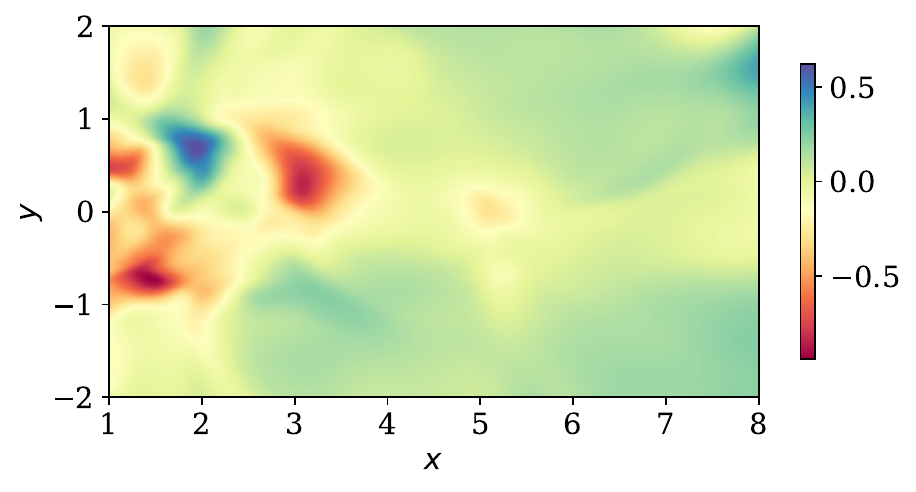}{(g) R-PINN}
\hfill
\figpanel{0.235\linewidth}{0.16\textheight}{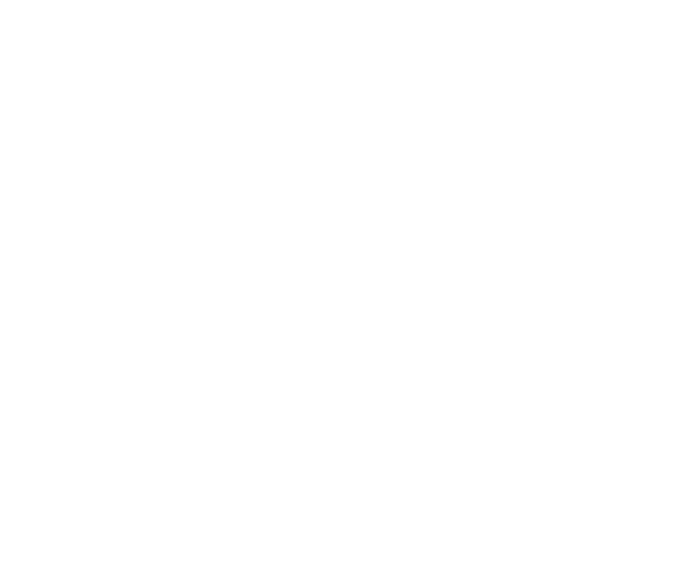}{\phantom{(h)}}

\caption{Representative reconstructions for the NSInv pressure field $p$. Panels (a)--(c) show the results of CFM, PG-CFM, and PG-CFM-ERFM, respectively; panel (d) shows the ground truth; panels (e)--(g) show PINN, B-PINN, and R-PINN.}
\label{fig:nsinv-qualitative}
\end{figure}

\subsection{Additional inverse problems and robustness}
\label{sec:additional-robustness}

\subsubsection{Additional inverse problems}
\label{subsubsec:additional-problems}

We also report results on HInv, BInv, and WInv; see
Tables~\ref{tab:hinv_metrics}--\ref{tab:winv_mse}. For these problems,
PG-CFM-ERFM gives lower reported errors than PG-CFM and CFM. It also improves
over the grid-based generative baselines in the reported field- or
parameter-recovery metrics.

On HInv, PG-CFM-ERFM gives the most accurate coefficient recovery across all reported metrics. Relative to PG-CFM, it reduces L2RE from $9.28\mathrm{E}{-}2$ to $5.17\mathrm{E}{-}2$ and MSE from $3.64\mathrm{E}{-}2$ to $1.13\mathrm{E}{-}2$. The gap relative to CFM is larger still, which suggests that purely data-driven transport is particularly sensitive to corruption. The grid-based generative baselines also perform substantially worse.

On BInv and WInv, where the inverse task is to recover scalar parameters, PG-CFM-ERFM again gives the lowest MSE and the smallest or near-smallest variance. The PIDM and PBFM baselines are less accurate, and the PINN-based baselines also remain behind the proposed method. Taken together, these results indicate that the combination of pathwise physics guidance and energy-based refinement remains effective across elliptic, parabolic, and hyperbolic inverse problems.

\begin{table}[t]
\centering
\caption{Mean $\pm$ standard deviation of error metrics for HInv coefficient recovery.}
\label{tab:hinv_metrics}
\setlength{\tabcolsep}{2pt}
\scalebox{0.8}{
\begin{tabular}{lccccccccc}
\toprule
Metric
& PG-CFM-ERFM
& PG-CFM
& CFM
& PINN
& R-PINN
& B-PINN
& PIDM-ME
& PIDM-SE
& PBFM \\
\midrule
L2RE
& \best{5.17E-2} & 9.28E-2 & 1.78E-1 & 1.31E-1 & 1.37E-1 & 1.92E-1 & 3.05E-1 & 4.65E-1 & 2.66E-1 \\
& {\scriptsize $\pm$4.73E-4} & {\scriptsize $\pm$1.33E-3} & {\scriptsize $\pm$2.30E-2}
& {\scriptsize $\pm$1.16E-2} & {\scriptsize $\pm$8.86E-3} & {\scriptsize $\pm$1.74E-2}
& {\scriptsize $\pm$1.51E-1} & {\scriptsize $\pm$1.00E-1} & {\scriptsize $\pm$3.00E-2} \\
\midrule
L1RE
& \best{3.95E-2} & 6.56E-2 & 1.45E-1 & 1.01E-1 & 1.08E-1 & 1.60E-1 & 2.52E-1 & 4.18E-1 & 2.17E-1 \\
& {\scriptsize $\pm$3.92E-4} & {\scriptsize $\pm$8.64E-4} & {\scriptsize $\pm$2.45E-2}
& {\scriptsize $\pm$1.06E-2} & {\scriptsize $\pm$7.89E-3} & {\scriptsize $\pm$1.63E-2}
& {\scriptsize $\pm$1.28E-1} & {\scriptsize $\pm$1.08E-1} & {\scriptsize $\pm$1.93E-2} \\
\midrule
MSE
& \best{1.13E-2} & 3.64E-2 & 1.37E-1 & 7.31E-2 & 7.96E-2 & 1.58E-1 & 4.92E-1 & 9.61E-1 & 3.04E-1 \\
& {\scriptsize $\pm$2.16E-4} & {\scriptsize $\pm$1.01E-3} & {\scriptsize $\pm$3.39E-2}
& {\scriptsize $\pm$1.25E-2} & {\scriptsize $\pm$1.05E-2} & {\scriptsize $\pm$2.80E-2}
& {\scriptsize $\pm$3.83E-1} & {\scriptsize $\pm$3.66E-1} & {\scriptsize $\pm$7.03E-2} \\
\midrule
MAE
& \best{7.88E-2} & 1.31E-1 & 1.33E+0 & 2.03E-1 & 2.17E-1 & 3.19E-1 & 5.05E-1 & 8.35E-1 & 1.23E+0 \\
& {\scriptsize $\pm$8.50E-4} & {\scriptsize $\pm$1.61E-3} & {\scriptsize $\pm$5.02E-1}
& {\scriptsize $\pm$2.12E-2} & {\scriptsize $\pm$1.58E-2} & {\scriptsize $\pm$3.26E-2}
& {\scriptsize $\pm$2.56E-1} & {\scriptsize $\pm$2.17E-1} & {\scriptsize $\pm$2.08E-1} \\
\bottomrule
\end{tabular}}
\end{table}

\begin{table}[t]
\centering
\caption{Mean $\pm$ standard deviation of MSE for identifying the viscosity $\nu$ on BInv.}
\label{tab:binv_mse}
\setlength{\tabcolsep}{2pt}
\scalebox{0.8}{
\begin{tabular}{lccccccccc}
\toprule
Method
& PG-CFM-ERFM
& PG-CFM
& CFM
& PINN
& R-PINN
& B-PINN
& PIDM-ME
& PIDM-SE
& PBFM \\
\midrule
MSE
& \best{3.63E-5} & 1.89E-4 & 3.96E-3 & 1.96E-4 & 2.21E-4 & 2.01E-4 & 3.72E-1 & 3.53E-1 & 1.31E-1 \\
& {\scriptsize $\pm$2.77E-5} & {\scriptsize $\pm$1.65E-4} & {\scriptsize $\pm$3.51E-3}
& {\scriptsize $\pm$1.74E-4} & {\scriptsize $\pm$1.63E-4} & {\scriptsize $\pm$1.85E-4}
& {\scriptsize $\pm$3.00E-4} & {\scriptsize $\pm$4.15E-3} & {\scriptsize $\pm$1.70E-1} \\
\bottomrule
\end{tabular}}
\end{table}

\begin{table}[t]
\centering
\caption{Mean $\pm$ standard deviation of MSE for identifying the wave speed $c$ on WInv.}
\label{tab:winv_mse}
\setlength{\tabcolsep}{2pt}
\scalebox{0.8}{
\begin{tabular}{lccccccccc}
\toprule
Method
& PG-CFM-ERFM
& PG-CFM
& CFM
& PINN
& R-PINN
& B-PINN
& PIDM-ME
& PIDM-SE
& PBFM \\
\midrule
MSE
& \best{5.47E-3} & 1.91E-2 & 1.45E+0 & 3.95E-2 & 1.00E-2 & 2.78E-2 & 4.91E-1 & 5.20E-1 & 7.46E-1 \\
& {\scriptsize $\pm$5.03E-3} & {\scriptsize $\pm$8.15E-3} & {\scriptsize $\pm$1.76E-2}
& {\scriptsize $\pm$2.03E-3} & {\scriptsize $\pm$5.36E-3} & {\scriptsize $\pm$2.35E-3}
& {\scriptsize $\pm$1.23E-3} & {\scriptsize $\pm$4.52E-3} & {\scriptsize $\pm$2.70E-1} \\
\bottomrule
\end{tabular}}
\end{table}

\subsubsection{Robustness under noise and corruption}
\label{subsubsec:robustness}

We next report robustness results under several corruption settings and
control comparisons. Unless otherwise stated, all results are shown as mean
$\pm$ standard deviation over three random seeds in terms of L2RE.

 Figure~\ref{fig:ablation-ab}(a) compares different placements of the physics
terms, including PIDM, PBFM, CFM, CFM-$\ell$,
CFM-g, and PG-CFM. In this
experiment, the local PDE weight is fixed at \(0.001\). PG-CFM gives the lowest error and the most stable performance across the tested settings, which indicates that bridge-local regularization and domain-wide collocation provide complementary information.  Figure~\ref{fig:ablation-ab}(b)
compares PG-CFM-ERFM with continued-training and label-refinement controls. The proposed method consistently outperforms these alternatives across PInv and NSInv, showing that the Stage-2 gain is not explained by additional optimization alone but by the energy-based reweighting itself.

\begin{figure}[t]
  \centering
  \figpanel{0.48\linewidth}{0.23\textheight}{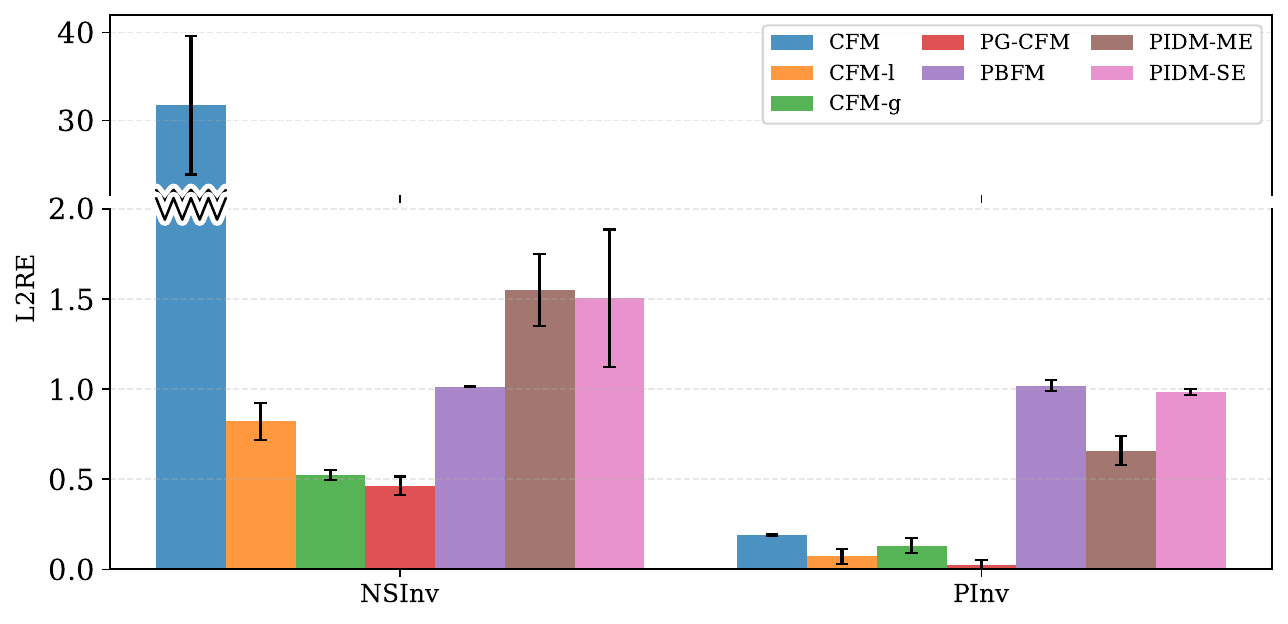}{(a) Physics placement comparison}
  \hfill
  \figpanel{0.48\linewidth}{0.23\textheight}{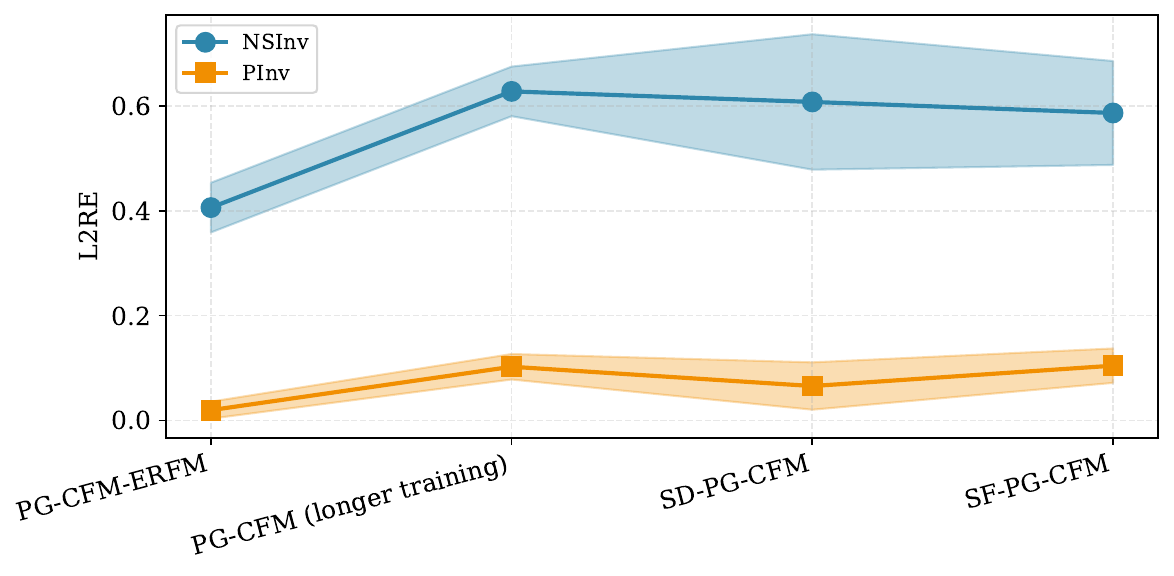}{(b) Compute-matched controls}
  \caption{Ablations A and B. Panel (a) compares different ways of incorporating physics into the flow model. Panel (b) compares ERFM with continued-training and label-refinement controls.}
  \label{fig:ablation-ab}
\end{figure}

Figures~\ref{fig:ablation-cd}(a) and~\ref{fig:ablation-cd}(b) show the effect
of varying the corruption ratio $\rho$ and the corrupted-noise scale
$\sigma_{\mathrm{bad}}$ on NSInv. In the corruption-ratio sweep,
PG-CFM-ERFM gives lower error than PG-CFM for $\rho\le 0.8$, while the gap
disappears at $\rho=1.0$, where all observations are corrupted. This is consistent with the fact that, in this extreme case, no clean observation subset remains, making it difficult for the Stage-1 energy to identify reliably trustworthy samples. In the
noise-scale sweep, PG-CFM-ERFM remains consistently more accurate than
PG-CFM over the tested range of $\sigma_{\mathrm{bad}}$. In both sweeps, the
PINN-based baselines have substantially larger errors and variability.

Finally, we evaluate robustness under different noise laws, including Gaussian, Laplacian, Student-$t$, uniform, and outlier-mixture noise. Figure~\ref{fig:ablation-e} shows that PG-CFM-ERFM is the most stable method across all tested distributions. In particular, the method remains stable under heavy-tailed Student-$t$ noise, where large outliers occur more frequently than under Gaussian corruption. These results indicate that the reweighting criterion is not restricted to the Gaussian corruption model used in the main comparisons.

\begin{figure}[t]
  \centering
  \figpanel{0.48\linewidth}{0.24\textheight}{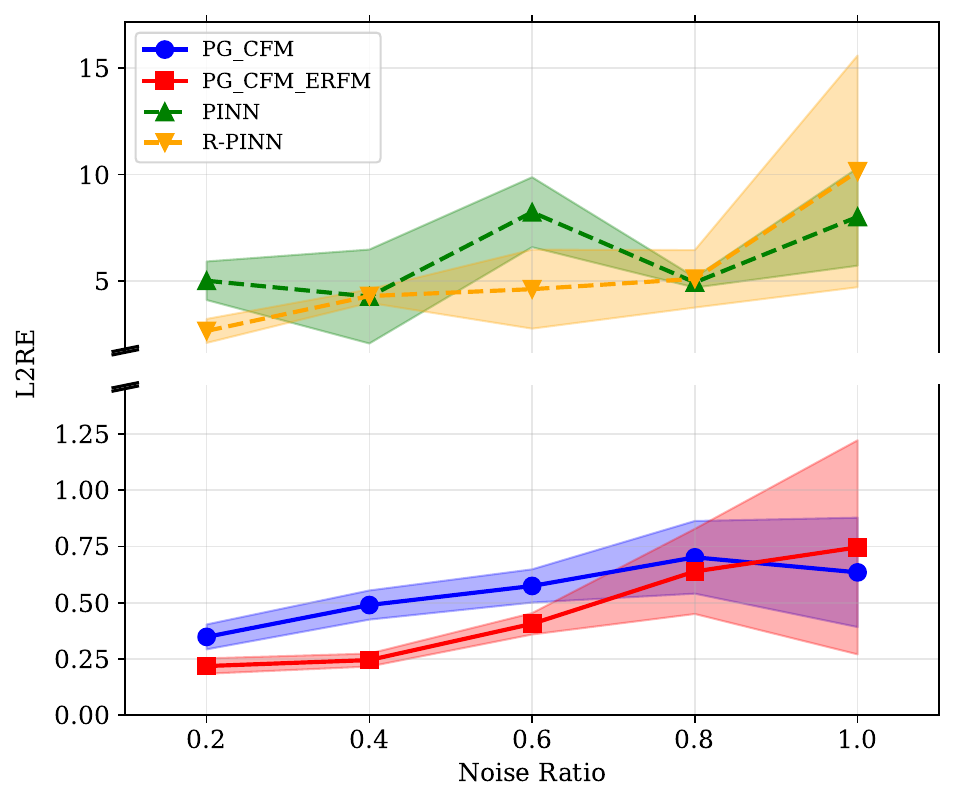}{(a) Outlier fraction sweep $\rho$}
  \hfill
  \figpanel{0.48\linewidth}{0.24\textheight}{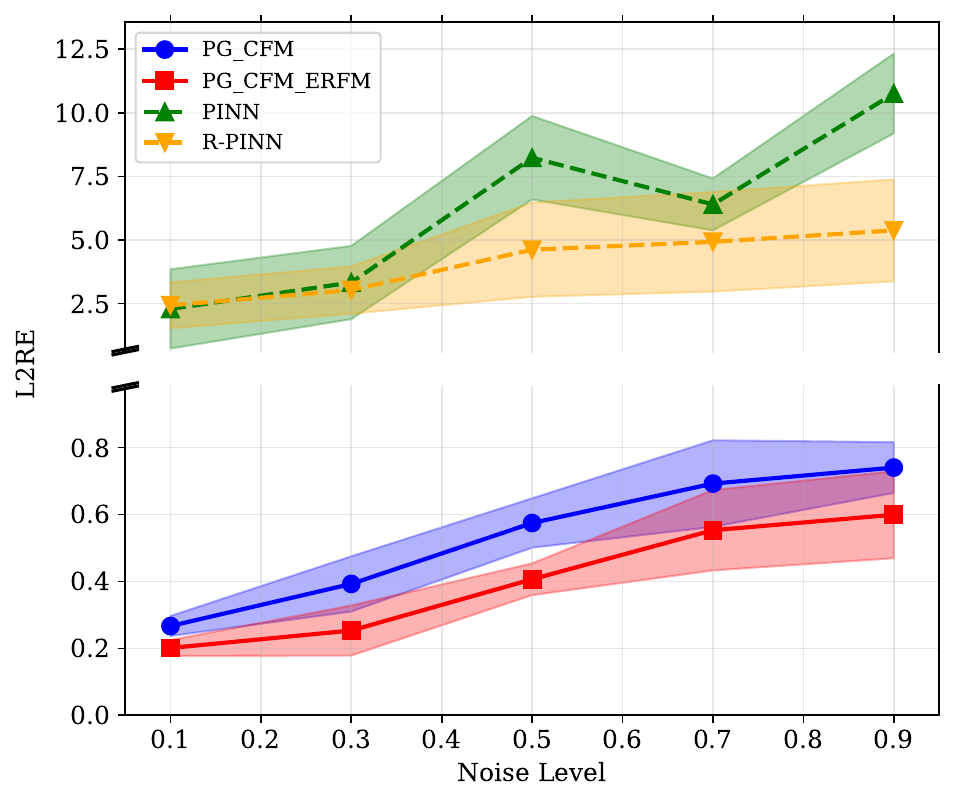}{(b) Outlier scale sweep $\sigma_{\mathrm{bad}}$}
  \caption{Ablations C and D: robustness to corruption magnitude on NSInv.}
  \label{fig:ablation-cd}
\end{figure}

\begin{figure}[t]
  \centering
  \includegraphics[width=0.6\linewidth]{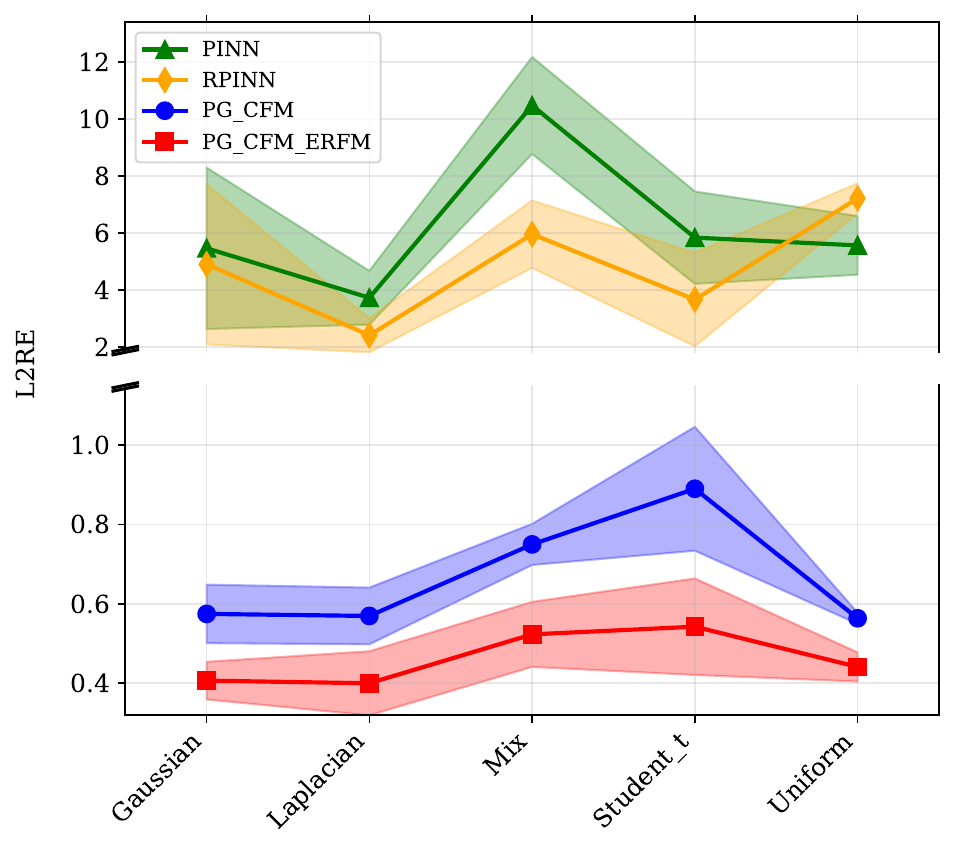}
  \caption{Ablation E: performance across different noise distributions on NSInv.}
  \label{fig:ablation-e}
\end{figure}

\subsection{Sensitivity and computational diagnostics}
\label{sec:sensitivity-diagnostics}

\subsubsection{Sensitivity to ERFM hyperparameters}
\label{subsubsec:erfm-hyper}

We examine the sensitivity of ERFM to its main hyperparameters, namely the inverse-temperature $\lambda$ in the weighting function $w(E)=\sigma(-\lambda E)$ and the threshold parameter $\kappa$ used in the robust energy normalization. Figure~\ref{fig:ablation-hyperparameter} reports
the NSInv results for sweeps of both parameters.

% The method remains stable across a broad range of settings. Smaller values of $\kappa$ tend to yield lower reconstruction error, which is consistent with a stricter robust normalization that suppresses high-energy observations more aggressively. As $\lambda$ increases from small values, the error decreases rapidly and then levels off. Overall, the
% results indicate that the second-stage reweighting is not highly sensitive to
% moderate changes in these parameters.

The method is reasonably stable over the tested range. Smaller values of
$\kappa$ tend to give lower reconstruction error, while the error decreases as
$\lambda$ increases from small values and then levels off. Overall, the
results indicate that the second-stage reweighting is not highly sensitive to
moderate changes in these parameters.

\begin{figure}[t]
\centering
\figpanel{0.48\linewidth}{0.21\textheight}{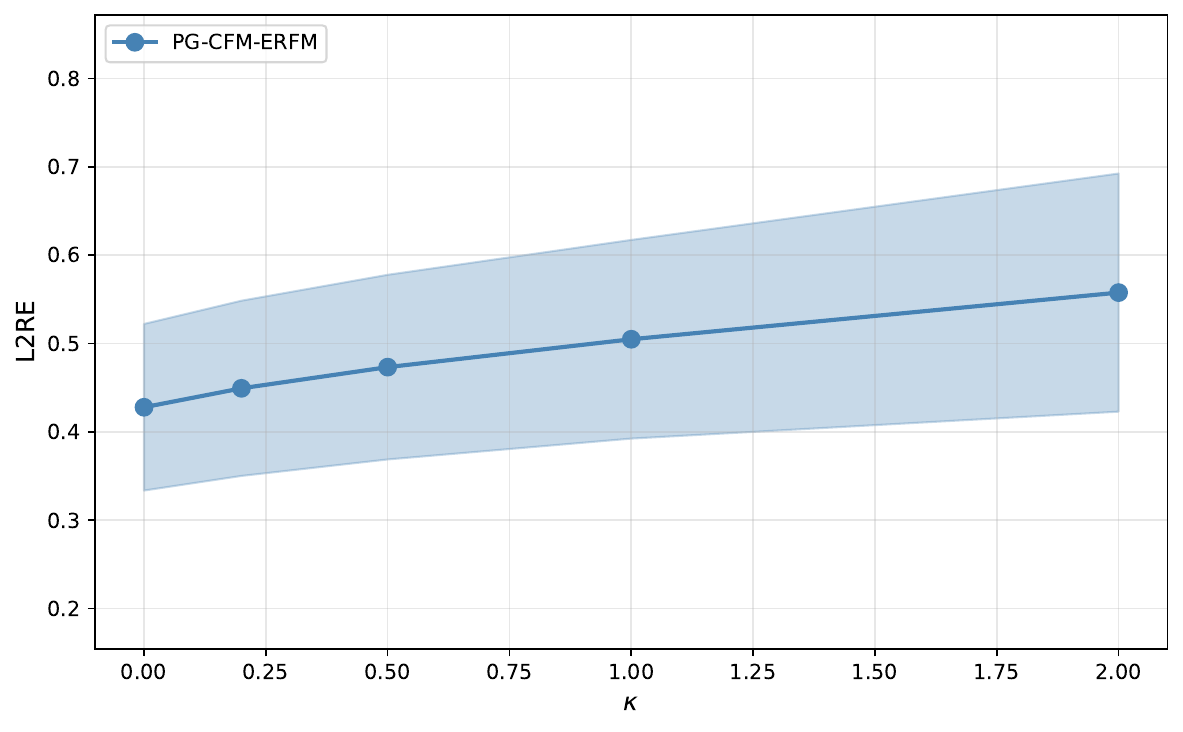}{(a) $\kappa$ sweep}
\hfill
\figpanel{0.48\linewidth}{0.21\textheight}{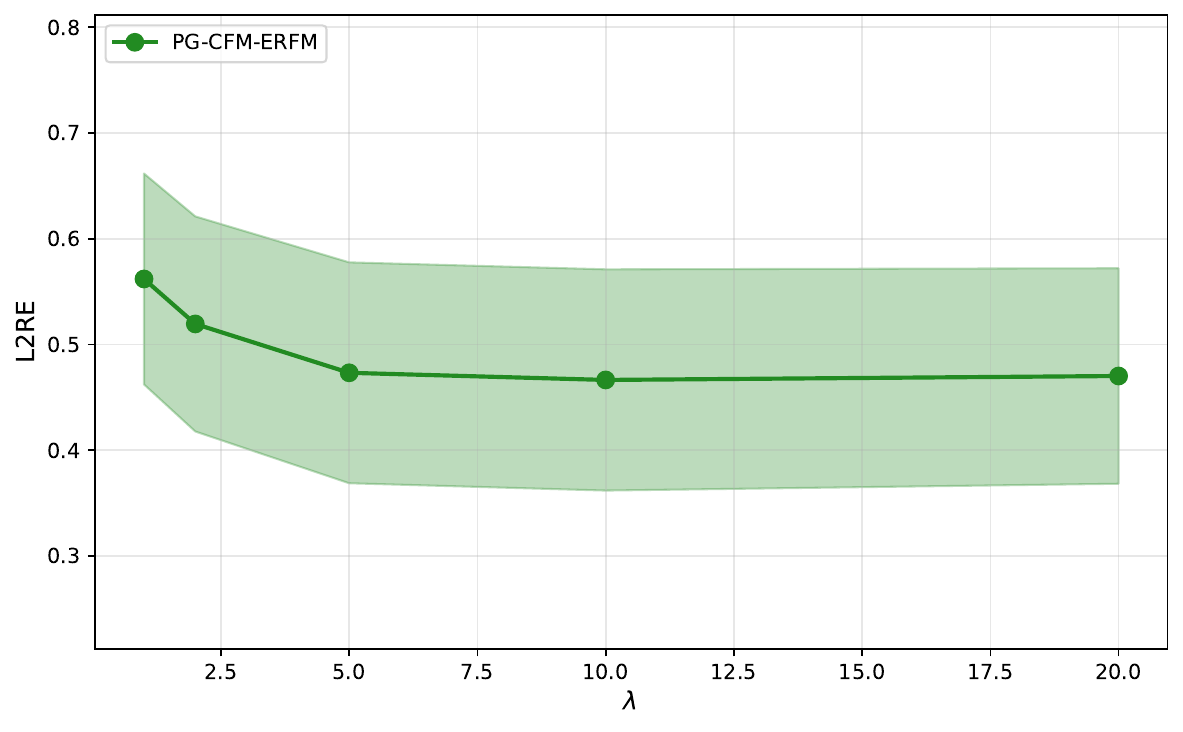}{(b) $\lambda$ sweep}
\caption{Sensitivity of ERFM to the hyperparameters $\kappa$ and $\lambda$ on NSInv. Reconstruction error (L2RE, mean $\pm$ standard deviation over three seeds) is reported for sweeps of the robust-normalization threshold $\kappa$ and the inverse-temperature $\lambda$.}
\label{fig:ablation-hyperparameter}
\end{figure}

\subsubsection{Sensitivity to numerical integration and residual evaluation}
\label{subsubsec:numerical-diagnostics}

We next study the sensitivity of the method to the numerical approximation used in the local and global physics losses. Figure~\ref{fig:ablation-e3} shows the effect of varying the Heun step counts $K_1$ and $K_2$ in NSInv. Across the tested values, PG-CFM-ERFM remains more accurate than PG-CFM, and
both methods show similar trends as the integration depth changes. This suggests that the gain from the second stage is not tied to a specific choice of Heun step count.

To further assess whether the Stage-2 weighting is physically meaningful, Figure~\ref{fig:ablation-energy-data} compares the ground-truth observation error with the energy score produced by the frozen Stage-1 model. The high-energy regions largely overlap with the strongly corrupted observations, which supports using the physics--data energy as a sample-reweighting criterion in this experiment. Since the score is computed from the frozen Stage-1 model, it combines
observation misfit and PDE residual information, providing a joint diagnostic
of data discrepancy and physics inconsistency.

\begin{figure}[t]
\centering
\figpanel{0.48\linewidth}{0.21\textheight}{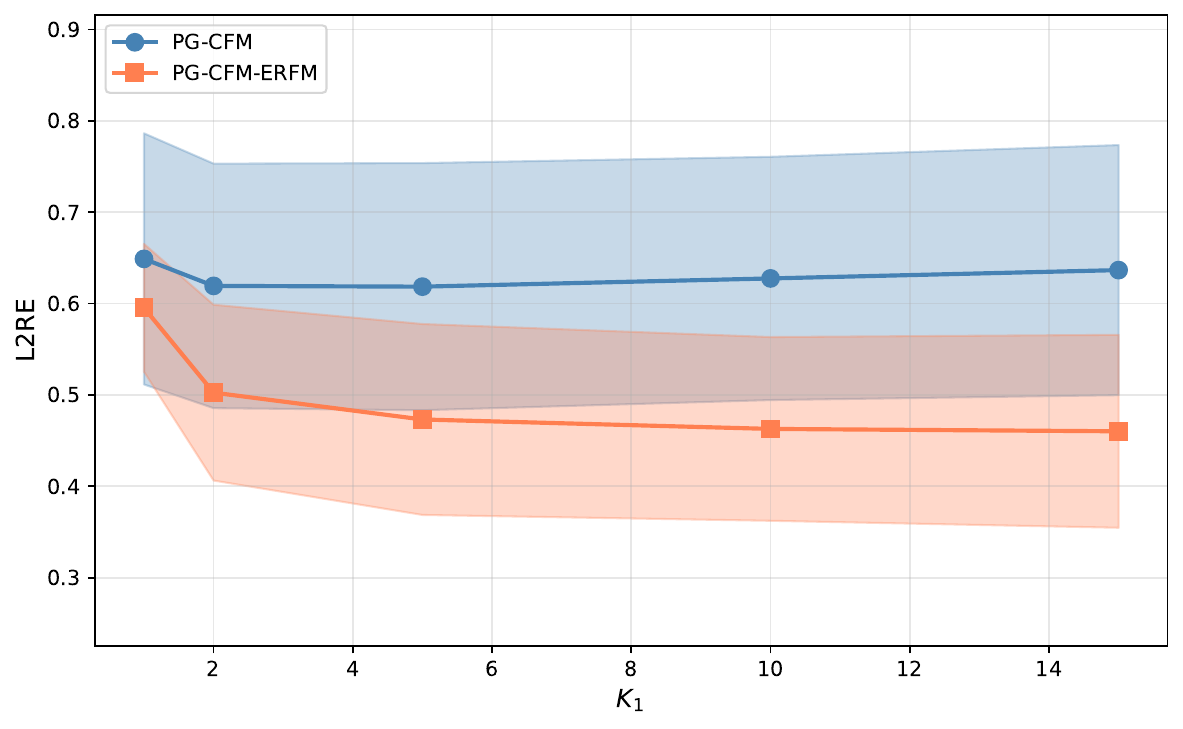}{(a) $K_1$ sweep}
\hfill
\figpanel{0.48\linewidth}{0.21\textheight}{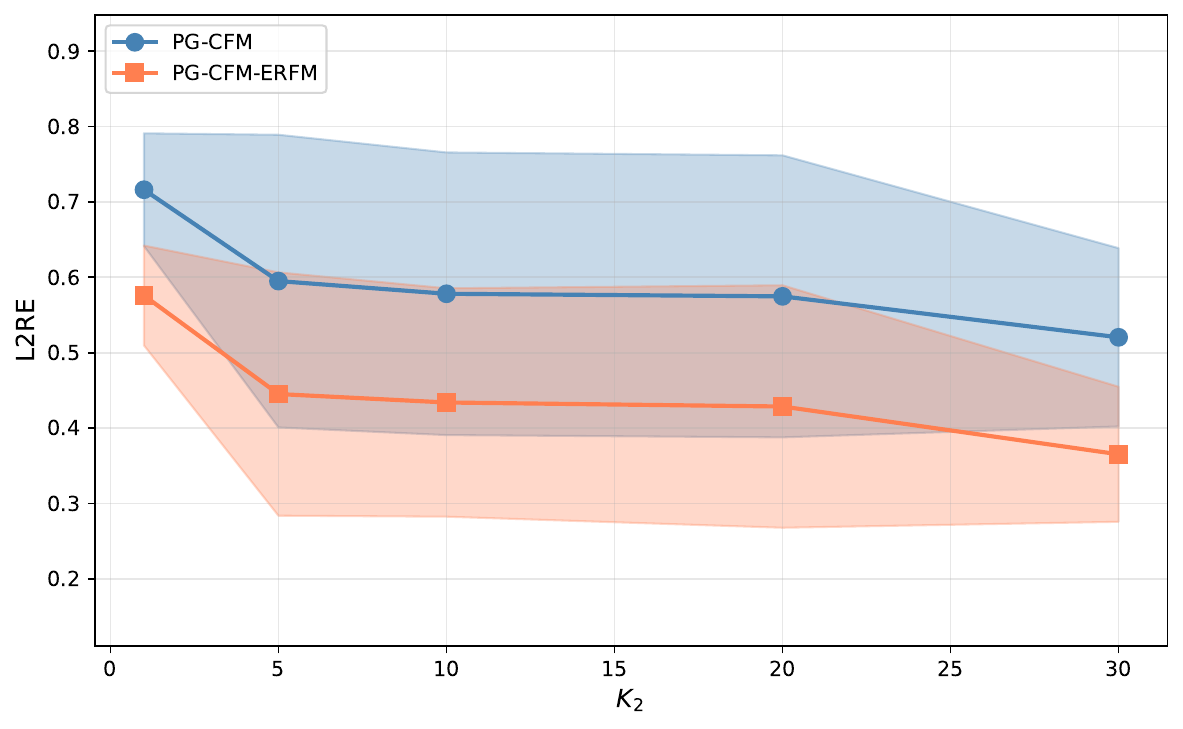}{(b) $K_2$ sweep}
\caption{Sensitivity to the Heun step counts used in the local and global physics losses on NSInv. L2RE (mean $\pm$ standard deviation) is shown as a function of $K_1$ and $K_2$.}
\label{fig:ablation-e3}
\end{figure}

\begin{figure}[t]
\centering
\figpanel{0.48\linewidth}{0.21\textheight}{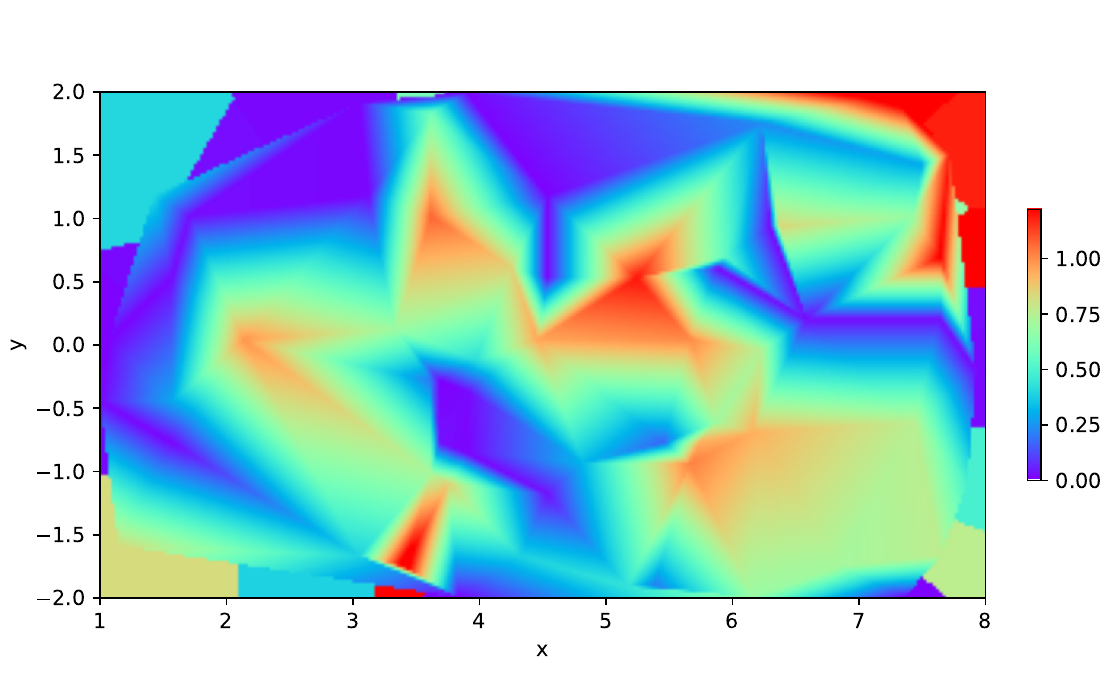}{(a) Ground-truth observation error $|\varepsilon|$}
\hfill
\figpanel{0.48\linewidth}{0.21\textheight}{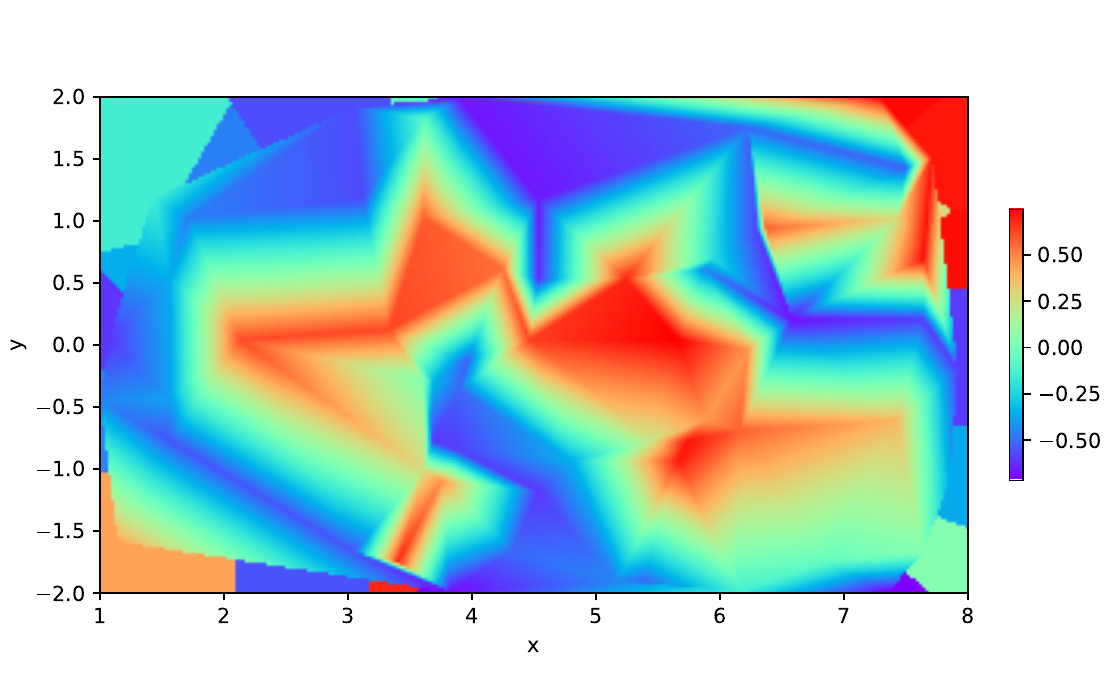}{(b) Stage-1 energy score $E$}
\caption{Comparison between observation error and Stage-1 energy scores on NSInv. High-energy regions align with strongly corrupted observations.}
\label{fig:ablation-energy-data}
\end{figure}

\subsubsection{Computational cost}
\label{subsubsec:cost}

\begin{table}[t]
\centering
\caption{NSInv: runtime and peak GPU memory of representative learned methods on an NVIDIA RTX 3090 Ti GPU. Reported training times correspond to the cost required to reach the final result for each method rather than to a common epoch or step budget.}
\label{tab:cost_pinn_vs_pgcfm_erfm}
\setlength{\tabcolsep}{3.5pt}
\small
\begin{tabular}{l|cccc}
\toprule
\multirow{2}{*}{Method} & \multicolumn{2}{c|}{Time (s)} & \multicolumn{2}{c}{Peak GPU memory (MB)} \\
\cmidrule(lr){2-3} \cmidrule(lr){4-5}
& Train & Infer & Train & Infer \\
\midrule
PG-CFM-ERFM      & 1.32E+3 & 4.24E-3 & 1.44E+3 & 1.91E+2 \\
PG-CFM    & 1.01E+3 & 4.13E-3 & 1.44E+3 & 1.57E+2 \\
CFM       & 1.97E+1 & 4.84E-3 & 1.24E+2 & 3.32E+2 \\
PINN      & 1.20E+3 & 5.12E-3 & 1.56E+2 & 2.86E+2 \\
R-PINN    & 8.48E+2 & 5.09E-3 & 1.56E+2 & 2.86E+2 \\
B-PINN    & 6.63E+2 & 2.36E-1 & 1.64E+2 & 4.88E+2 \\
PBFM      & 2.16E+2 & 9.66E-3 & 1.08E+2 & 1.08E+2 \\
PIDM-ME   & 6.72E+1 & 1.20E-1 & 2.63E+3 & 1.03E+3 \\
PIDM-SE   & 5.41E+1 & 7.43E-2 & 2.65E+3 & 1.04E+3 \\
\bottomrule
\end{tabular}
\end{table}

% Table~\ref{tab:cost_pinn_vs_pgcfm_erfm} reports the runtime and peak GPU
% memory on NSInv. The reported training times correspond to the cost required
% for each method to reach its final result, rather than to a common epoch or
% step budget. Under this comparison, PG-CFM-ERFM is more expensive than CFM
% and several lightweight baselines, but remains broadly comparable in training
% time to the PINN-based methods while giving lower reconstruction error under
% heavy corruption.

% ERFM adds only a moderate overhead over PG-CFM, which is consistent with its
% role as a short Stage-2 refinement rather than a second full training
% procedure. By contrast, CFM is computationally cheap but much less accurate on
% NSInv. Relative to PINN, R-PINN, and B-PINN, PG-CFM-ERFM has a comparable
% order of training cost while achieving lower reconstruction error.

% The grid-based generative baselines PBFM and PIDM also require less training
% time than PG-CFM-ERFM, but are less accurate in this sparse, noisy, off-grid
% setting. In addition, the PIDM variants use substantially more GPU memory.
% Overall, the proposed two-stage method does not have the lowest computational
% cost, but remains competitive in accuracy among the more effective baselines.

Table~\ref{tab:cost_pinn_vs_pgcfm_erfm} reports the runtime and peak GPU memory on NSInv. The reported training times correspond to the cost required for each method to reach its final result, rather than to a common epoch or step budget. PG-CFM-ERFM is more expensive than CFM and the grid-based generative baselines, but its training time is of the same order as the PINN-based methods. ERFM adds only a moderate overhead over PG-CFM, which is consistent with its
role as a short Stage-2 refinement rather than a second full training
procedure. 

These results should be interpreted together with Tables~\ref{tab:ns_metrics} and~\ref{tab:nsinv_beta_mse}. CFM and the grid-based generative baselines are cheaper in this experiment but give larger reconstruction or coefficient errors under the tested corrupted observations, while PG-CFM-ERFM trades additional training cost for improved inverse recovery.

\section{Conclusion}
\label{sec:conclusion}

We developed a two-stage, mesh-free flow-matching method for PDE inverse problems from sparse, noisy, and corrupted observations. The first stage, PG-CFM, introduces physics-guided conditional flow matching by combining observation-conditioned transport with strong-form residual penalties along short solver-integrated trajectories and intermittent global collocation constraints. The second stage, ERFM, uses a frozen Stage-1 teacher to construct physics--data energy scores and reweights the observation-conditioned flow-matching loss to reduce the influence of PDE-inconsistent samples. We showed that ERFM is equivalent to observation-conditioned flow matching under a teacher-induced reweighted empirical distribution, which explains Stage~2 as a distributional reweighting step rather than a full retraining procedure.

Across five benchmarks with heterogeneous corruption and heavy-tailed noise, PG-CFM-ERFM consistently improves field and coefficient recovery over PINN-based baselines, standard CFM, and competing generative baselines. The ablation studies further show that the improvement from ERFM cannot be attributed to additional optimization alone, and that bridge-local and global physics losses play complementary roles in the Stage-1 model.

The current formulation has two main limitations. First, the physics-guided losses rely on fixed-step explicit Heun integration, which is efficient but may become unreliable in stiff regimes. Second, the effectiveness of ERFM depends on the ability of the Stage-1 teacher to produce informative energy scores; under extremely severe corruption, the separation between relatively reliable and corrupted observations may deteriorate. Future work will consider implicit or adaptive integrators for the pathwise physics losses, more complex boundary and initial operators and multi-physics couplings, and larger-scale evaluations in higher-dimensional or experimental settings with stronger model mismatch and structured corruption.

\FloatBarrier
\bibliographystyle{siamplain}
\bibliography{main}

\begin{thebibliography}{10}

\bibitem{arridge2019solving}
{\sc S.~Arridge, P.~Maass, O.~{\"O}ktem, and C.-B. Sch{\"o}nlieb}, {\em Solving inverse problems using data-driven models}, Acta Numerica, 28 (2019), pp.~1--174.

\bibitem{baldan2025flow}
{\sc G.~Baldan, Q.~Liu, A.~Guardone, and N.~Thuerey}, {\em Flow matching meets pdes: A unified framework for physics-constrained generation}, arXiv preprint arXiv:2506.08604,  (2025).

\bibitem{bastek2025physics}
{\sc J.-H. Bastek, W.~Sun, and D.~Kochmann}, {\em Physics-informed diffusion models}, in The Thirteenth International Conference on Learning Representations, 2025.

\bibitem{basu2021influence}
{\sc S.~Basu, P.~Pope, and S.~Feizi}, {\em Influence functions in deep learning are fragile}, in International Conference on Learning Representations, 2021.

\bibitem{baydin2018automatic}
{\sc A.~G. Baydin, B.~A. Pearlmutter, A.~A. Radul, and J.~M. Siskind}, {\em Automatic differentiation in machine learning: a survey}, Journal of machine learning research, 18 (2018), pp.~1--43.

\bibitem{ben2024d}
{\sc H.~Ben-Hamu, O.~Puny, I.~Gat, B.~Karrer, U.~Singer, and Y.~Lipman}, {\em D-flow: Differentiating through flows for controlled generation}, in Forty-first International Conference on Machine Learning, 2024.

\bibitem{benning2018modern}
{\sc M.~Benning and M.~Burger}, {\em Modern regularization methods for inverse problems}, Acta numerica, 27 (2018), pp.~1--111.

\bibitem{chen2018neural}
{\sc R.~T. Chen, Y.~Rubanova, J.~Bettencourt, and D.~K. Duvenaud}, {\em Neural ordinary differential equations}, Advances in neural information processing systems, 31 (2018).

\bibitem{cheng2025gradientfree}
{\sc C.~Cheng, B.~Han, D.~C. Maddix, A.~F. Ansari, A.~Stuart, M.~W. Mahoney, and B.~Wang}, {\em Gradient-free generation for hard-constrained systems}, in The Thirteenth International Conference on Learning Representations, 2025.

\bibitem{chung2023diffusion}
{\sc H.~Chung, J.~Kim, M.~T. Mccann, M.~L. Klasky, and J.~C. Ye}, {\em Diffusion posterior sampling for general noisy inverse problems}, in The Eleventh International Conference on Learning Representations, 2023.

\bibitem{dhariwal2021diffusion}
{\sc P.~Dhariwal and A.~Nichol}, {\em Diffusion models beat gans on image synthesis}, Advances in neural information processing systems, 34 (2021), pp.~8780--8794.

\bibitem{duan2025copinn}
{\sc S.~Duan, W.~Wu, P.~Hu, Z.~Ren, D.~Peng, and Y.~Sun}, {\em Co{PINN}: Cognitive physics-informed neural networks}, in Forty-second International Conference on Machine Learning, 2025.

\bibitem{ghosh2017robust}
{\sc A.~Ghosh, H.~Kumar, and P.~S. Sastry}, {\em Robust loss functions under label noise for deep neural networks}, in Proceedings of the AAAI conference on artificial intelligence, vol.~31, 2017.

\bibitem{ho2020denoising}
{\sc J.~Ho, A.~Jain, and P.~Abbeel}, {\em Denoising diffusion probabilistic models}, Advances in neural information processing systems, 33 (2020), pp.~6840--6851.

\bibitem{huang2024diffusionpde}
{\sc J.~Huang, G.~Yang, Z.~Wang, and J.~J. Park}, {\em Diffusion{PDE}: Generative {PDE}-solving under partial observation}, in The Thirty-eighth Annual Conference on Neural Information Processing Systems, 2024.

\bibitem{isakov2017inverse}
{\sc V.~Isakov}, {\em Inverse Problems for Partial Differential Equations}, vol.~127, Springer, 2017.

\bibitem{karniadakis2021physics}
{\sc G.~E. Karniadakis, I.~G. Kevrekidis, L.~Lu, P.~Perdikaris, S.~Wang, and L.~Yang}, {\em Physics-informed machine learning}, Nature Reviews Physics, 3 (2021), pp.~422--440.

\bibitem{li2025videopde}
{\sc E.~Li, Z.~Wang, J.~Huang, and J.~J. Park}, {\em Videopde: Unified generative pde solving via video inpainting diffusion models}, arXiv preprint arXiv:2506.13754,  (2025).

\bibitem{lipman2023flow}
{\sc Y.~Lipman, R.~T.~Q. Chen, H.~Ben-Hamu, M.~Nickel, and M.~Le}, {\em Flow matching for generative modeling}, in The Eleventh International Conference on Learning Representations, 2023.

\bibitem{liu2025rethinking}
{\sc S.~Liu, Y.~Yao, J.~Jia, S.~Casper, N.~Baracaldo, P.~Hase, Y.~Yao, C.~Y. Liu, X.~Xu, H.~Li, et~al.}, {\em Rethinking machine unlearning for large language models}, Nature Machine Intelligence,  (2025), pp.~1--14.

\bibitem{liu2025threats}
{\sc Z.~Liu, H.~Ye, C.~Chen, Y.~Zheng, and K.-Y. Lam}, {\em Threats, attacks, and defenses in machine unlearning: A survey}, IEEE Open Journal of the Computer Society,  (2025).

\bibitem{meyer2021alternative}
{\sc G.~P. Meyer}, {\em An alternative probabilistic interpretation of the huber loss}, in Proceedings of the ieee/cvf conference on computer vision and pattern recognition, 2021, pp.~5261--5269.

\bibitem{pruthi2020estimating}
{\sc G.~Pruthi, F.~Liu, S.~Kale, and M.~Sundararajan}, {\em Estimating training data influence by tracing gradient descent}, Advances in Neural Information Processing Systems, 33 (2020), pp.~19920--19930.

\bibitem{raissi2019physics}
{\sc M.~Raissi, P.~Perdikaris, and G.~E. Karniadakis}, {\em Physics-informed neural networks: A deep learning framework for solving forward and inverse problems involving nonlinear partial differential equations}, Journal of Computational physics, 378 (2019), pp.~686--707.

\bibitem{ren2018learning}
{\sc M.~Ren, W.~Zeng, B.~Yang, and R.~Urtasun}, {\em Learning to reweight examples for robust deep learning}, in International conference on machine learning, PMLR, 2018, pp.~4334--4343.

\bibitem{simone2025continual}
{\sc L.~Simone, D.~Bacciu, and S.~Ma}, {\em Continualflow: Learning and unlearning with neural flow matching}, in ICML 2025 Workshop on Machine Unlearning for Generative AI, 2025.

\bibitem{song2022learning}
{\sc H.~Song, M.~Kim, D.~Park, Y.~Shin, and J.-G. Lee}, {\em Learning from noisy labels with deep neural networks: A survey}, IEEE transactions on neural networks and learning systems, 34 (2022), pp.~8135--8153.

\bibitem{song2021scorebased}
{\sc Y.~Song, J.~Sohl-Dickstein, D.~P. Kingma, A.~Kumar, S.~Ermon, and B.~Poole}, {\em Score-based generative modeling through stochastic differential equations}, in International Conference on Learning Representations, 2021.

\bibitem{tauberschmidt2025physics}
{\sc J.~Tauberschmidt, S.~Fellenz, S.~J. Vollmer, and A.~B. Duncan}, {\em Physics-constrained fine-tuning of flow-matching models for generation and inverse problems}, arXiv preprint arXiv:2508.09156,  (2025).

\bibitem{utkarsh2025physics}
{\sc U.~Utkarsh, P.~Cai, A.~Edelman, R.~Gomez-Bombarelli, and C.~V. Rackauckas}, {\em Physics-constrained flow matching: Sampling generative models with hard constraints}, in The Thirty-ninth Annual Conference on Neural Information Processing Systems, 2025.

\bibitem{wang2021understanding}
{\sc S.~Wang, Y.~Teng, and P.~Perdikaris}, {\em Understanding and mitigating gradient flow pathologies in physics-informed neural networks}, SIAM Journal on Scientific Computing, 43 (2021), pp.~A3055--A3081.

\bibitem{xia2022sample}
{\sc X.~Xia, T.~Liu, B.~Han, M.~Gong, J.~Yu, G.~Niu, and M.~Sugiyama}, {\em Sample selection with uncertainty of losses for learning with noisy labels}, in International Conference on Learning Representations, 2022.

\bibitem{xiang2022self}
{\sc Z.~Xiang, W.~Peng, X.~Liu, and W.~Yao}, {\em Self-adaptive loss balanced physics-informed neural networks}, Neurocomputing, 496 (2022), pp.~11--34.

\bibitem{yang2021b}
{\sc L.~Yang, X.~Meng, and G.~E. Karniadakis}, {\em B-pinns: Bayesian physics-informed neural networks for forward and inverse pde problems with noisy data}, Journal of Computational Physics, 425 (2021), p.~109913.

\bibitem{yao2025guided}
{\sc J.~Yao, A.~Mammadov, J.~Berner, G.~Kerrigan, J.~C. Ye, K.~Azizzadenesheli, and A.~Anandkumar}, {\em Guided diffusion sampling on function spaces with applications to {PDE}s}, in The Thirty-ninth Annual Conference on Neural Information Processing Systems, 2025.

\bibitem{ye2025pdeformer}
{\sc Z.~Ye, Z.~Liu, B.~Wu, et~al.}, {\em Pdeformer-2: A versatile foundation model for two-dimensional partial differential equations}, arXiv preprint arXiv:2507.15409,  (2025).

\bibitem{yuan2025pirf}
{\sc M.~Yuan, P.~Jin, N.~Li, and Q.~Li}, {\em {PIRF}: Physics-informed reward fine-tuning for diffusion models}, in NeurIPS 2025 AI for Science Workshop, 2025.

\bibitem{zhang2018generalized}
{\sc Z.~Zhang and M.~Sabuncu}, {\em Generalized cross entropy loss for training deep neural networks with noisy labels}, Advances in neural information processing systems, 31 (2018).

\bibitem{zhang2020distilling}
{\sc Z.~Zhang, H.~Zhang, S.~O. Arik, H.~Lee, and T.~Pfister}, {\em Distilling effective supervision from severe label noise}, in Proceedings of the IEEE/CVF Conference on Computer Vision and Pattern Recognition, 2020, pp.~9294--9303.

\bibitem{zhou2025text2pde}
{\sc A.~Zhou, Z.~Li, M.~Schneier, J.~R.~B. Jr, and A.~B. Farimani}, {\em Text2{PDE}: Latent diffusion models for accessible physics simulation}, in The Thirteenth International Conference on Learning Representations, 2025.

\bibitem{zhou2024data}
{\sc W.~Zhou and Y.~Xu}, {\em Data-guided physics-informed neural networks for solving inverse problems in partial differential equations}, arXiv preprint arXiv:2407.10836,  (2024).

\end{thebibliography}

\end{document}